\documentclass[final,leqno]{siamltex}
\usepackage[hidelinks,pagebackref=true,pdfauthor={Andrew}]{hyperref}

\usepackage{cite}

\usepackage{xcolor}
\definecolor{dark-red}{rgb}{0.4,0.15,0.15}
\definecolor{dark-blue}{rgb}{0.15,0.15,0.4}
\definecolor{medium-blue}{rgb}{0,0,0.5}
\hypersetup{
    colorlinks, linkcolor={blue},
    citecolor={blue}, urlcolor={dark-red}
}

\usepackage{mathptmx}
\usepackage{float}
\usepackage[section,subsection]{extraplaceins}

\usepackage{amssymb}
\usepackage{amsmath}
\usepackage{graphicx}
\usepackage{enumerate}
\usepackage{subfigure}

\usepackage{color}
\usepackage{wrapfig}
\usepackage{bigints}
\usepackage{stmaryrd}

\usepackage[final, noinline]{fixme}

\def\smartqed{}
\def\qed{}

\def\norm#1{\|#1\|}

\def\R{\mathbb{R}} 
\def\T{\mathbb{T}} 
\def\P{\mathbb{P}}

\usepackage{mathtools}
\usepackage[hidelinks]{hyperref}
\makeatletter
\newcommand{\mytag}[2]{%
  \text{#1}%
  \@bsphack
  \begingroup
    \@onelevel@sanitize\@currentlabelname
    \edef\@currentlabelname{%
      \expandafter\strip@period\@currentlabelname\relax.\relax\@@@%
    }%
    \protected@write\@auxout{}{%
      \string\newlabel{#2}{%
        {#1}%
        {\thepage}%
        {\@currentlabelname}%
        {\@currentHref}{}%
      }%
    }%
  \endgroup
  \@esphack
}
\makeatother

\let\norm\relax
\newcommand{\Norm}[1]{\left\lVert #1 \right\rVert}
\newcommand{\norm}[1]{\left\lvert #1 \right\rvert}

\newcommand\restr[2]{{
  \left.\kern-\nulldelimiterspace 
  #1 
  \vphantom{\big|} 
  \right|_{#2} 
  }}

\newtheorem{remark}[theorem]{Remark}

\title{Homogenization of lateral diffusion on a random surface}

\author{A. B. Duncan}

\begin{document}

\maketitle

\begin{abstract}
	We study the problem of lateral diffusion on a static, quasi-planar surface generated by a stationary, ergodic random field possessing rapid small-scale spatial fluctuations.  The aim is to study the effective behaviour of a particle undergoing Brownian motion on the surface viewed a projection on the underlying plane.  By formulating the problem as a diffusion in a random medium, we are able to use known results from the theory of stochastic homogenization of SDEs to show that, in the limit of small scale fluctuations, the diffusion process behaves quantitatively like a Brownian motion with constant diffusion tensor $D$.  While $D$ will not have a closed-form expression in general,  we are able to derive variational bounds for the effective diffusion tensor, and using a duality transformation argument, obtain a closed form expression for $D$ in the special case where $D$ is isotropic.  We also describe a numerical scheme for approximating the effective diffusion tensor numerically and illustrate this scheme with two examples.
\end{abstract}

\begin{keywords}
	stochastic homogenisation, random media, Laplace-Beltrami, diffusion, multiscale-analysis.
\end{keywords}

\begin{AMS}
	35Q92, 60H30, 35B27
\end{AMS}

\pagestyle{myheadings}
\thispagestyle{plain}
\markboth{A. B. DUNCAN}{Homogenization of lateral diffusion on a random surface} 

\section{Introduction}
\label{sec:introduction}
Lateral diffusion of particles along interfaces is a frequently occurring phenomenon in cellular biology.  In the case of lipid bilayer membranes,  the lipid molecules and integral membrane proteins which constitute the cell membrane themselves undergo lateral diffusion as a result of thermal agitation \cite{almeida1995lateral}.   The mobility of membrane proteins has far-reaching implications for many cellular processes, in particular protein transport, signalling and morphology \cite{axelrod1976mobility, alberts2000molecular, bressloff2013stochastic},  thus there has been considerable interest in measuring how the mobility of these proteins is affected by the membrane and its surrounding environment.  The dynamics of protein diffusion within a fluid membrane was first considered by Saffman and Delbr\"{u}ck \cite{saffman1975brownian} who proposed a continuum hydrodynamic model for a laterally diffusing particle in a flat, homogeneous fluid membrane.  The model predicted a relationship for the diffusion tensor of a particle in the membrane in terms of the particle radius, the thickness and viscosity of the membrane and the viscosity of the bulk medium.  Continuum models for studying the influence of shape fluctuations of the membrane on the macroscopic protein diffusion rate were subsequently considered in \cite{gustafsson1997diffusion, granek1997semi, naji2007diffusion,reister2007lateral} in which the proteins undergo Brownian motion laterally along an infinitesimally thin two dimensional surface embedded in $\R^3$.  The equilibrium fluctuations of the surface are characterised by the Canhan-Helfrich Hamiltonian \cite{helfrich1973elastic,canham1970minimum}.  In \cite{naji2007diffusion}, an expression for the effective diffusion tensor was derived by considering the joint Markov process for the coupled particle and surface, applying an adiabatic elimination to average out the rapid temporal surface fluctuations. Another approach, based on a path-integral formulation was considered in \cite{reister2007lateral}.  More recently, in \cite{duncan2013multiscale}, a multiscale approach to the problem of lateral diffusion on rapidly fluctuating surfaces was adopted and expressions for the macroscopic diffusion tensor were derived rigorously,  firstly on a static surface with periodic undulations,  and subsequently on a surface possessing both rapid spatial and temporal oscillations. 
\\\\
In this paper we build on the work of \cite{duncan2013multiscale}, moving on from the case of periodic media to allow random surface fluctuations generated by a stationary, spatially ergodic random field. For simplicity, we restrict our attention to the static membrane model.  The resulting model is very general and thus applicable to a wide variety of surfaces containing inhomogeneities and micro-structure, making the approach especially attractive for biological applications.  By viewing the system as diffusion in a  random medium, under reasonable assumptions, it is possible to apply stochastic homogenization methods to  derive expressions for the effective diffusion tensor for such a model,  and in many cases, obtain a closed form expression.  To our knowledge, the study of lateral diffusion on random surfaces with spatially ergodic fluctuations has not been considered previously, either analytically or numerically.  The novelty of this paper thus lies in the application of standard results from the theory of stochastic homogenization and random media to analyse the dynamics of this model rather than any particular mathematical result.
\\\\
Consider a random field $h^\epsilon(x)$ which describes the surface fluctuations about the plane,  where the small scale parameter $\epsilon \ll 1$ controls the small scale amplitude and wavelength of the surface fluctuations.   In Section \ref{sec:model} we will describe the model for lateral diffusion of a particle on this rapidly-fluctuating, random surface. Moreover, we will show that the trajectory $X^\epsilon(t)$ of such a particle can be described by an It\^{o} SDE with rapidly varying, random coefficients and with a singularly perturbed drift term.  The problem of identifying the macroscopic behaviour of the projected trajectory $X^\epsilon(t)$ in the limit as $\epsilon \rightarrow 0$ is a homogenization problem.
\\\\
Homogenization of parabolic and elliptic problems with random, stationary coefficients has been widely studied, both from a PDE perspective  \cite{papanicolaou1979boundary, papanicolaou1982diffusions} as well as from a probabilistic perspective \cite{kipnis1986central, de1989invariance, osada1983homogenization}.  In this paper we approach this problem probabilistically, and in Section \ref{sec:formulation} we formulate this system as a stochastic homogenization problem using the framework of \cite{komorowski2012fluctuations}.  Using methods from the theory of stochastic homogenization for SDEs one can then identify the limiting behaviour of the diffusion process.  Indeed, in Section \ref{sec:homogenization}, under certain reasonable assumptions on the surface fluctuations, we show that, in the limit of vanishing $\epsilon$, the evolution of the particle is well approximated by a pure diffusion process on the plane with constant effective diffusion tensor $D$. 
\\\\
As in the periodic case,  $D$ will not generally have a closed form in two dimensions.  In Section \ref{sec:properties} we show that it is possible to generalize the results of \cite[Proposition 2]{duncan2013multiscale} and express $D$ in terms of a variational minimisation problem, from which Voigt-Reuss variational bounds \cite{zhikov1994homogenization} on $D$ can be derived.  In Section \ref{sec:area_scaling}, by generalising the results of \cite[Section 5.3]{duncan2013multiscale} we apply a duality transformation argument to show that for two-dimensional surfaces, if $D$ is isotropic then it is equal to $\frac{1}{Z}$ where $Z$ is the average  surface area of the random surface with respect to its projection on the plane. This is a generalisation of the \emph{area scaling approximation} described in \cite{gustafsson1997diffusion,naji2009hybrid,king2004apparent,gov2006diffusion}.  Moreover, we identify a natural sufficient condition for the effective diffusion tensor $D$ to be isotropic. In particular, we show that it is sufficient for the random field to be isotropic itself for the area scaling approximation to hold.
\\\\
When the effective diffusion tensor is not isotropic, then one must resort to numerical methods to compute $D$.  Unlike in the periodic case \cite[Section 5]{duncan2013multiscale}, the expression for the effective diffusion tensor is not amenable to direct numerical approximation.   In Section \ref{sec:numerical_methods} we describe a well-known approach to computing the effective diffusion tensor via a periodization approximation \cite{owhadi2003approximation, bourgeat2004approximations, alexanderian2012homogenization}. This allows us to approximate the solution of the infinite cell problem with the solution of a periodic cell problem over a suitably large domain in $\R^d$.  
\\\\
We apply this scheme to two particular examples. First we consider a random protrusion model, where the random surface is generated by randomly placed protrusions,  whose position is determined by a Poisson point process. This model falls under the framework discussed in the previous sections, and we demonstrate that the area scaling approximation holds for this example. In the second example we consider lateral diffusion on a random surfaces defined by the graph of sufficiently smooth Gaussian random field. Due to the unboundedness of the fluctuations this example will not fall under the above theory, however, numerical simulations suggest that homogenization limit does appear to exist for this particular model, and moreover, the area scaling approximation holds all the same. 
\\\\
In Section \ref{sec:conclusion} we provide concluding remarks as well as suggestions for future avenues of research.

\section{Model}
\label{sec:model}
In this section we introduce the model for lateral diffusion on a rough, random surface.  For simplicity, we will we restrict our attention to surfaces $S$ which can be expressed as the graph of a sufficiently smooth function $h:\R^d \rightarrow \R$ known in the biology literature as the \emph{Monge gauge} for  $S$.  We will model the rough interface as a surface $S^\epsilon$ consisting of low amplitude, high frequency undulations about the plane.  More specifically, for a small scale parameter $\epsilon \ll 1$, we consider a surface $S_h^\epsilon$ with Monge Gauge:
\begin{equation}
	\label{eq:h_eps}
	h^\epsilon(x) = \epsilon h\left(\frac{x}{\epsilon}\right), \qquad \mbox{ for } x \in \R^d,
\end{equation}
so that
$$S^\epsilon_h = \left\lbrace \left(x,  h^\epsilon\left(x\right)\right) \, \Big| \, x \in \R^d 	\right\rbrace.$$
\\
We assume that the function $h(x)$ is a random field with measure $\P$ having mean $0$ and being stationary, that is, having two-point covariance function of the form:
$$\mathbb{E}_{\P}\left[h(x)h(y)\right]  = C(x-y), \qquad x, y \in \R^d,$$
for some positive function $C$.  Moreover, we will assume that the random field is ergodic with respect to spatial translations, so that expectations with respect to $\P$ can be replaced by spatial averages.  Finally, we assume that realisations of $h(x)$ are $\P$-almost surely bounded with (sufficiently many) bounded derivatives (which precludes the possibility of Gaussian random fields).
\\\\
In local coordinates,  the surface $S^\epsilon_h$ has metric tensor $g^\epsilon(x, h) = g(x/\epsilon, h)$, where $g$ is given by:
\begin{equation}
	\label{eq:metric_tensor}
	g(x, h) = I + \nabla h(x) \otimes \nabla h(x), 	\qquad x \in \R^d,
\end{equation}
and the infinitesimal surface element is given by $\sqrt{\norm{g}\left(\frac{x}{\epsilon}\right)}$ where $\norm{g}$ denotes the determinant of $g$.  Since the random field is ergodic with respect to spatial translations the average surface area $Z$ can be written as
\begin{equation}
	\label{eq:Z}
	Z := \mathbb{E}_{\mathbb{P}}\left[\sqrt{\norm{g}(x,h)}\right] = \lim_{R\rightarrow \infty} \frac{1}{(2R)^d}\int_{[-R,R]^d} \sqrt{\norm{g}(x,h)} \, dx.
\end{equation}
\fxnote{Should I add a plot here of increasingly rescaled functions long with the arc-length to emphasise this result, or is the text self-explanatory?}
In particular,  for fixed $\epsilon$ the average surface area is given by
\begin{equation*}
	\lim_{R\rightarrow \infty} \frac{1}{(2R)^d}\int_{[-R,R]^d} \sqrt{\norm{g^\epsilon}(x,h)} \, dx = \lim_{R\rightarrow \infty} \left(\frac{2R}{\epsilon}\right)^{-d}\int_{\left[-R/\epsilon,R/\epsilon\right]^d} 
\sqrt{\norm{g}(y, h)} \, dy  = Z.
\end{equation*}
This implies that as $\epsilon\rightarrow 0$, the surface area is conserved, which suggests that (\ref{eq:h_eps}) is the natural scaling for this problem. This is illustrated for the 1D case in Figure \ref{fig:one_dim}, which plots a realisation of the surface generated by a Gaussian random field $h^\epsilon(x)$.  The arc-length of the surface over $[-R,R]$, for $R \gg 1$, is  approximately  $2R Z$. Consider the projected trajectory of a particle undergoing lateral diffusion on $S^\epsilon_h$ starting from $0$. The escape time of the process from $[-R,R]$ is equal to the expected escape time of a free $\R$-valued Brownian motion from the interval $[-RZ,RZ]$ which is $\frac{R^2Z^2}{2}$.  Taking $\epsilon \rightarrow 0$, the expected escape time remains $\frac{R^2Z^2}{2}$ in the limit, which implies that the law of the lateral diffusion process behaves identically to a free Brownian motion on $\R$ with constant diffusion coefficient $\frac{1}{Z^2}$ . We note that any other scaling would result in the surface area going to $0$ or $\infty$ as $\epsilon \rightarrow 0$.  It follows that the scaling given in (\ref{eq:h_eps}) preserves the average surface area which suggests that (\ref{eq:h_eps}) is the correct scaling for this problem.
\begin{figure}[ht]
\includegraphics[width=\textwidth]{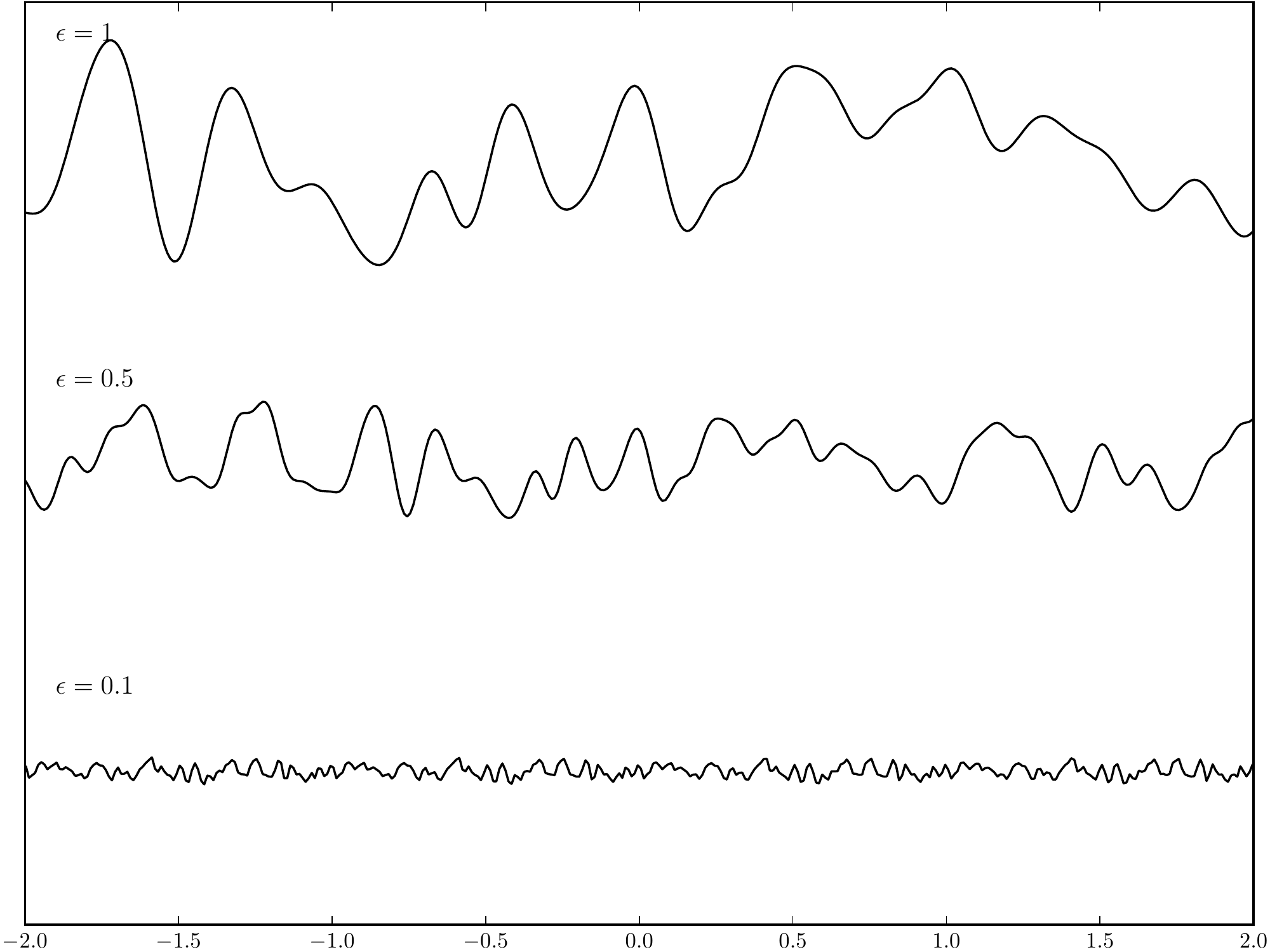}
\caption[Realisation of the 1D surface generated by a Gaussian random field]{Realisation of the 1D surface generated by a Gaussian random field $h^\epsilon$ for increasingly small values of $\epsilon$.  As $\epsilon \rightarrow 0$, the average arclength $Z$ is remains constant at around $6.5$. }
\label{fig:one_dim}
\end{figure}

Consider a particle diffusing along the surface $S_h^\epsilon$ and let  $X^\epsilon(t)$ denote the position of the particle at time time given in local coordinates, with $X^\epsilon(0) = x$. Then, the evolution of $X^\epsilon(t)$ is described by the following It\^{o} SDE
\begin{equation}
	\label{eq:ergodic_sde}
	dX_h^\epsilon(t) = \frac{1}{\epsilon}F(X_h^\epsilon(t)/\epsilon,h) \,dt + \sqrt{2\Sigma(X_h^\epsilon(t)/\epsilon,h)}\,dB(t),
\end{equation}
where 
\begin{equation}
	\label{eq:drift}
	F(x,h) = \frac{1}{\sqrt{\norm{g}(x,h)}}\nabla_x\cdot\left(\sqrt{\norm{g}(x,h)}g^{-1}(x,h)\right),
\end{equation}
and \begin{equation}
	\label{eq:diffusion}
	\Sigma(x,h) = g^{-1}(x,h).
\end{equation}
\\
Equivalently,  consider  an observable $u_h^\epsilon(x, t)$ of $X^\epsilon_h(t)$ defined by
$$
	u_h^\epsilon(x, t) = \mathbb{E}\left[u( X^\epsilon_h(t) \, | \, X^\epsilon_h(0) = x\right],
$$
where $u \in C_b(\R^d)$.  Then $u_h^\epsilon(x,t)$ satisfies the following backward Kolmogorov equation \cite[Chapter 6]{friedman1975stochastic}:
\begin{equation}
\label{eq:ergodic_pde}
\begin{aligned}
	\frac{\partial u_h^\epsilon(x,t)}{\partial t} &= \mathcal{L}_h^\epsilon u_h^\epsilon(x,t), 	\quad &(x,t) \in \R^d \times (0,T], \\
	u_h^\epsilon(x, t ) &= u(x),	\quad &(x,t) \in \R^d\times\lbrace 0 \rbrace.
\end{aligned}
\end{equation}
where 
\begin{equation}
	\mathcal{L}_h^\epsilon f(x) = \frac{1}{\sqrt{\norm{g}(x/\epsilon,h)}}\nabla_x\cdot\left(\sqrt{\norm{g}(x/\epsilon,h)}g^{-1}(x/\epsilon)\nabla_x f(x)\right),
\end{equation} 
\\
Our objective is to show that as $\epsilon \rightarrow 0$,  the process $X_h^\epsilon(t)$ behaves quantitively like a Brownian motion with a constant effective diffusion tensor $D$ independent of the particular realisation of $h$.  Equivalently,  we show that $u_h^\epsilon$ converges pointwise to the solution $u^0$ of the PDE

\begin{equation}
\label{eq:ergodic_eff_pde}
\begin{aligned}
	\frac{\partial u^0(x,t)}{\partial t} &= D:\nabla_x\nabla_x u^0(x,t), 	\quad &(x,t) \in \R^d \times (0,T], \\
	u^0(t,x ) &= u(x),	\quad &(x,t) \in \R^d\times\lbrace 0 \rbrace.
\end{aligned}
\end{equation}

\section{Problem Formulation and Set-up}
\label{sec:formulation}
In this section we will rigorously state the assumptions on the random field $h(x)$ which are necessary for the problem to be well-defined and for a homogenization limit of both the SDE (\ref{eq:ergodic_sde}) and the PDE (\ref{eq:ergodic_pde}) to exist.  The approach described here is a direct application of the results in \cite[Chapter 9]{komorowski2012fluctuations},  whose approach we will follow very closely.
\\\\
Let $\Omega$ be the space of all $C^3$ functions from $\R^d$ to $\R$ equipped with the Fr\'{e}chet metric generated by seminorms of the form $$\Norm{f}_{N} =  \displaystyle\sup_{\norm{x} \leq N}\sum_{k \leq 3}\norm{\nabla^{k} f(x)}, \qquad N \in \mathbb{N}.$$  Equipped with this metric, one can show that  $\Omega$ is a Polish space.
\\\\
For $x \in \R^d$, define the translation operator $\tau_x:\Omega \rightarrow \Omega$ by $$\tau_x(h) = h(\cdot + x), \qquad h \in \Omega.$$  Let $\P$ be a Borel probability measure on the measurable space $(\Omega, \mathcal{B}(\Omega))$ and define the group of translations $\lbrace \tau_x \, : \, x \in \R^d \rbrace$.  We assume that the following conditions hold on the random field $h(x)$:
\\
\begin{description}
	\item [\textbf{A}.] $\P\left(\tau^{-1}_x\left(B\right)\right) = \P\left(B\right)$, for all $B \in \mathcal{B}(\Omega)$ and $x \in \R^d$. (Stationarity)
	\item [\textbf{B}.] For  $B \in \mathcal{B}(\Omega)$,  $\tau_x(B) = B$ for all $x \in \R^d$ implies that $\P(B) = 0$ or $\P(B) = 1$, (Ergodicity)
	\item [\textbf{C}.] For all $\delta > 0$, $\lim_{x \rightarrow 0}\P\left[\norm{\tau_x h - h} >  \delta \right] = 0$. (Stochastic Continuity)
\end{description}
Moreover, we shall make the following assumption regarding the derivatives of realisations of $h$:
\begin{description}
	\item [\textbf{D}.] There exists a constant $K > 0$ such that for $\mathbb{P}$-almost surely every realisation of $h \in \Omega$,
	\begin{equation}
	\label{eq:ergodic_bounds_assumption}
		\norm{\nabla h(x)} + \norm{\nabla\nabla h(x)} \leq K,	\qquad x \in \R^d,  \quad \mathbb{P}-\mbox{a.s.}
	\end{equation}
\end{description}
Assumption \textbf{D} is a very restrictive assumption which precludes considering Gaussian random fields,  however,  without this assumption one encounters insurmountable technical problems when attempting to obtain a homogenization result.  
\\\\
We note that the scenario of diffusion on surfaces possessing static, periodic fluctuations, as considered in \cite{duncan2013multiscale} can be expressed in the framework described above.  Indeed, if $h_0 \in C^3(\T^d)$ extended to $\R^d$ by periodicity,  then we can define a random field by $$h(x) = h_0(x + \zeta),$$ where $\zeta$ is distributed according to the Lebesgue measure on $\T^d$.  The corresponding probability measure $\P$  on $\Omega$ clearly satisfies the conditions \textbf{A}-\textbf{D}, and moreover the SDE (\ref{eq:ergodic_sde}) and the PDE (\ref{eq:ergodic_pde}) reduce to their periodic counterparts for the periodic surface map $h_0$.
\\\\
Since Brownian motion is invariant under the diffusive scaling $t \rightarrow t/\epsilon^2$, $x \rightarrow x/\epsilon$ we can express the process $X^\epsilon_h(t)$ in law as
$$X^\epsilon_h(t) = \epsilon X_h\left(\frac{t}{\epsilon^2}\right),$$
where $X_h(t)$ is the solution of the It\^{o} SDE
\begin{equation}
\label{eq:ergodic_sde_rescaled}
	dX_h(t) = {F}(X_h(t), h)dt + \sqrt{2{\Sigma}(X_h(t), h)}\, dB(t),
\end{equation}
where $B(t)$ is a standard $\R^d$-valued Brownian motion.  
\\\\
For a fixed $h \in \Omega$ the infinitesimal generator of $X_h(t)$ is given by
\begin{equation}
	\mathcal{L}_h f = \frac{1}{\sqrt{\norm{{g}}(x, h)}}\nabla_x\cdot\left(\sqrt{\norm{{g}}(x, h)}{g}^{-1}(x, h)\nabla_x f(x)\right),  \qquad	f \in C^2_b(\R^d)
\end{equation}
First, we establish the well-posedness of the SDE for $X_h(t)$:
\\
\begin{proposition}
	Let $X_0$ be a random variable with finite second moments, independent of $B(\cdot)$ and the random field $h(x)$. Then, under assumption (\ref{eq:ergodic_bounds_assumption}), for $\mathbb{P}$-almost every $h \in \Omega$,   the SDE (\ref{eq:ergodic_sde_rescaled}) has a unique strong solution $X_h(t)$ satisfying $X_h(0) = X_0$.  Moreover, the $X_h(t)$ is a Markov diffusion process and possesses a  strictly positive continuous transition density $p(t, x, y, h)$.
\end{proposition}
\begin{proof}
Result follows from a direct application of Theorems 2.2, 3.6, 4.3 and    6.4.6 of \cite{friedman1975stochastic}.\\
\end{proof}

\section{The Homogenization Result}
\label{sec:homogenization}
In this section we state the homogenization result for the SDE (\ref{eq:ergodic_sde}) and PDE (\ref{eq:ergodic_pde}) making use of standard stochastic homogenization techniques such as \cite{bensoussan1978asymptotic,kipnis1986central,komorowski2012fluctuations}.  The approach adopted here closely follows that of \cite{komorowski2012fluctuations}.   The assumption that the random field $h$ is stationary and ergodic with respect to spatial translations is essential to obtaining a limiting diffusion process in the limit as $\epsilon \rightarrow 0$. To obtain such a homogenization limit, we need to express  SDE (\ref{eq:ergodic_sde_rescaled}) in terms of a stationary ergodic Markov process.  Following the work of  \cite{kipnis1986central,de1989invariance,papanicolaou1979boundary}, we considered the so-called \emph{environment viewed from the particle}.
\\\\
We first define the derivatives with respect to the translation group $\lbrace \tau_x \rbrace_{x \in \R^d}$, which are necessary for the formulation of the environment process.   For $i \in \lbrace 1, \ldots, d \rbrace$, let $D_i$ be the $L^2(\P)$ generator of $\tau_x$ in the $e_i$ direction,  that is $$D_i V := \frac{d}{d\lambda}V(\tau_{(\lambda e_i)} h)\big|_{\lambda = 0},$$
in the $L^2(\P)$ sense.  Assumption \textbf{C} permits us to apply Corollary 1.1.6 of \cite{ethier2009markov}, to show that the $\mathcal{D}(D_i)$ are dense in $L^2(\P)$.   Note that $D_i$ is antisymmetric with respect to the $L^2(\P)$ inner product, so that for all $U, V \in \mathcal{D}(D_i) \subset L^2(\P)$,
$$ \langle D_i U, V \rangle_{L^2(\P)} = -\langle U, D_i V \rangle_{L^2(\P)}.$$
For $V \in {H}^1 := \bigcap_{i=1}^d \mathcal{D}(D_i)$, we can then define the gradient to be 
\begin{equation}
\label{eq:ergodic_gradient}
\mathbf{\mathbb{D}} V := \left(D_i V\right)_{i=1}^d.
\end{equation}
For a vector field $\mathbf{V} = (V_i)_{i=1}^d$ such that $V_i \in H_1$ we define the divergence to be 
\begin{equation}
	\label{eq:ergodic_div}
	\mathbb{D}\cdot \mathbf{V} := \sum_{i=1}^d D_i V_i.
\end{equation}

We express the coefficients of the SDE as stationary random variables on $\Omega$.  Abusing notation, we define the random variable $g(h)$ by
$$g(h) := g(0, h) = \restr{I + \nabla h(x)\otimes \nabla h(x)}{x = 0},$$
we can express the coefficients of the SDE (\ref{eq:ergodic_sde_rescaled}) as random variables on $\Omega$.  Indeed, by defining
\begin{equation*} 
\begin{aligned}
F(h) := F(0, h) &= \restr{\frac{1}{\sqrt{\norm{g}(x,h)}}\nabla\cdot\left(\sqrt{\norm{g}(x,h)}g^{-1}(x, h)\right)}{x = 0} \\
		&= \frac{1}{\sqrt{\norm{g}(h)}}\,\mathbb{D}\cdot\left(\sqrt{\norm{g}(h)}g^{-1}(h)\right),
\end{aligned}
\end{equation*}
and,
$$\Sigma(h) := \Sigma(0, h) = g^{-1}(h),$$
we can then express (\ref{eq:ergodic_sde_rescaled}) as
$$
 dX_h(t) = F(\tau_{X_h(t)}h)\,dt + \sqrt{2\Sigma(\tau_{X_h(t)}h)}\,dB(t).
$$
Let $\zeta_h(t)$ be the stochastic process given by
\begin{equation*}
\zeta_h(t) = \begin{cases} \tau_{X_h(t)}h, & \mbox{if } t > 0 \\
			   h & \mbox{if } t = 0. \end{cases} 
\end{equation*}
This stationary, $\Omega$-valued stochastic process known as the \emph{environment viewed from the particle}, and was considered in works such as \cite{kipnis1986central,de1989invariance,papanicolaou1979boundary}.  It describes the evolution of the environment $h$ which is observed from a frame of reference fixed on the particle.  The process $\zeta_h(t)$ is Markovian and possesses an invariant measure $\pi$ absolutely continuous with respect to $\P$.  The particle trajectory $X_h$ is driven by $\zeta_h(t)$, in the sense that we can express $X_h(t)$ in terms of the environment process as follows 
$$
	X_h(t) = \int_0^t F(\zeta_h(s)) ds + \int_0^t \sqrt{2\Sigma(\zeta_h(s))}\, dB(s).
$$
$ $\newline
By assumption (\ref{eq:ergodic_bounds_assumption}) it follows that 
\begin{equation}
\label{eq:ergodic_Z}
Z = \int_{\Omega}\sqrt{\norm{g}(h)}\,\mathbb{P}(dh) = \int_{\Omega} \sqrt{1 + \norm{\nabla h(0)}^2} \, \mathbb{P}(dh) < \infty.
\end{equation}
  Define ${\pi}$ to be the probability measure on $h$ given by
$$
	{\pi}(dh) = \frac{\sqrt{\norm{g}(h)}}{Z}\mathbb{P}(dh)
$$

The following proposition summarises the properties of the environment process $\zeta_h(t)$ required to obtain a invariance principle for $X_h(t)$.  A proof of this result can be found in various places,  in particular  of \cite[Proposition 9.7]{komorowski2012fluctuations}.
\\
\begin{proposition}[Proposition 9.7, \cite{komorowski2012fluctuations}]
	The environment process $\zeta_h(t)$ is Markovian and its transition semigroup $P(t)$ can be written as
\begin{equation}
	\label{eq:ergodic_environment_semigroup}
	P(t)f(h) = \int_{\R^d}p(t,x, 0, h){f}(x, h) \,dx,	\qquad f \in L^\infty(\P)
\end{equation}
which can be extended to a positive preserving contraction semigroup on $L^p(\Omega)$ for any $p\geq 1$.  In particular
	$$\Norm{P(t)f}_{L^p(\pi)} \leq \Norm{f}_{L^p(\pi)},	\qquad f \in L^2(\pi).$$
	Moreover,  $\zeta_h(t)$ possesses an invariant measure $\pi$ with respect to which $\zeta_h(t)$ is reversible and ergodic.  Finally, the domain $C_b^2(\Omega)$ is a core for the $L^2$-generator $\mathcal{L}$ of $P(t)$ and $\mathcal{L}$ is the unique self-adjoint extension of 
	\begin{equation}
	\label{eq:periodic_generator_extension}
	\hat{\mathcal{L}}f = \frac{1}{\sqrt{\norm{g}(h)}}\mathbb{D}\cdot\left(\sqrt{\norm{g}(h)}g^{-1}(h)\mathbb{D} f\right), \qquad \mathcal{D}(\hat{L}) =  C^2_b(h),
	\end{equation}
and we can express the Dirichlet form corresponding to $\mathcal{L}$ as follows
	\begin{equation}
	\label{eq:ergodic_dirichlet_form}
		\langle (-\mathcal{L})f, f \rangle_{L^2(\pi)} = \frac{1}{Z}\int_{\Omega} \mathbb{D} f(h) \cdot g^{-1}(h) \mathbb{D} f(h) \sqrt{\norm{g}(h)} \mathbb{P}(dh).
	\end{equation}
\end{proposition}
\smartqed \qed
$ $ \linebreak
We now introduce the spaces $\mathcal{H}_1$ and its dual $\mathcal{H}_{-1}$ as defined in \cite{kipnis1986central} and \cite{de1989invariance}.  Let $\mathcal{H}_1$  be the completion of the space $$\left\lbrace \phi \in C^2_b(\Omega) \, \big| \, \int_{\Omega}\phi(h)\pi(dh) = 0 \mbox{ and }\Norm{\phi}_1 := \langle (-\mathcal{L})\phi, \phi\rangle_{L^2(\pi)} < \infty \right\rbrace,$$ with respect to $\Norm{\cdot}_1$.   The dual space $\mathcal{H}_{-1}$  is the completion of the space
$$\left\lbrace \phi \in C^2_b(\Omega) \, \big| \, \int_{\Omega}\phi(h)\pi(dh) = 0 \mbox{ and }\Norm{\phi}_{-1} < \infty \right\rbrace,$$ 
where the dual norm is given by $$\Norm{\phi}_{-1}^2 = \frac{1}{2}\langle \phi, (-\mathcal{L})^{-1}\phi \rangle_{L^2(\pi)} = \sup_{\psi \in H_1} \left\lbrace 2 \langle \phi, \psi \rangle - \langle (-\mathcal{L})\psi, \psi \rangle \right\rbrace.$$
Note that $\phi \in L^2(\pi)$ lies in $\mathcal{H}_{-1}$ if and only if there exists $C > 0$ such that
$$
	\langle \phi, \psi \rangle_{L^2(\pi)} \leq C \Norm{\psi}_1,
$$
for all $\psi \in \mathcal{H}_1$.   Moreover,  since $\mathcal{L}$ is positive, self-adjoint  $$\mathcal{H}_{1} = \mathcal{D}\left(\left({-L}\right)^{\frac{1}{2}}\right), \qquad \mbox{ and } \qquad \mathcal{H}_{-1} = \mathcal{D}\left(\left(-L\right)^{-\frac{1}{2}}\right).$$
By assumption (\ref{eq:ergodic_bounds_assumption}), the matrix $g(\omega)$ is uniformly elliptic.  This implies that 
\begin{equation*}
K_1\langle \mathbb{D}\phi, \mathbb{D}\phi \rangle_{L^2(\pi)}	\leq \langle \phi, \mathcal{L}\phi \rangle_{L^2(\pi)} \leq K_2\langle \mathbb{D}\phi, \mathbb{D}\phi \rangle_{L^2(\pi)}, \quad \mbox{ for } \phi \in C^1_b(\Omega),
\end{equation*}
for some positive constants $K_1$ and $K_2$.  This implies that there is an isomorphism between the spaces $\mathcal{H}_1$ and $H_1$, and thus, given $\phi \in \mathcal{H}_1$ we are justified in defining the gradient $\mathbb{D} \phi \in L^2(\Omega)$.
\\\\
Following the typical approach adopted in the homogenization of SDEs we wish to decompose the singularly perturbed drift term into a martingale and a remainder term which vanishes as $\epsilon \rightarrow 0$ and then apply the Martingale Central Limit Theorem \cite{helland1982central} to obtain convergence to a limiting Brownian motion.   Unlike in the periodic case,  due to the lack of a spectral gap (or equivalently of a Poincar\'{e} inequality) for $\mathcal{L}$, the Poisson problem $-\mathcal{L}\chi = F$ will not be well posed.  However, since the resolvent of $P(t)$ in $L^2(\pi)$ is $(0,\infty)$, for a fixed unit vector $e \in \R^d$ and $\lambda > 0$,  we can consider the following resolvent equation for $\chi^e \in L^2(\pi)$:
	\begin{equation}
		\label{eq:ergodic_cell_lambda}
		\left(\lambda I - \mathcal{L}\right) \chi^e = F^{e},
	\end{equation}
where $F^e = F\cdot e$.
\\
\begin{lemma}
	For any unit vector $e \in \R^d$, $$F^e \in L^2(\pi)\cap \mathcal{H}_{-1}$$
\end{lemma}
\begin{proof}
	To show that $F^e \in L^2(\pi)$, we note that
	$$ \norm{F(h)^e}  =  \norm{F(h)\cdot e} \leq C \norm{\nabla\nabla h(x)}_2,$$
	which is bounded almost surely, by assumption (\ref{eq:ergodic_bounds_assumption}).  To show that $F^e \in \mathcal{H}_{-1}$ we first note that the centering condition holds, so that
	$$ \int_{\Omega} F^e(h) \pi(dh) = 0. $$
Let $\psi \in \mathcal{H}_1$, then
	\begin{equation*}
	\begin{split}
		\langle F^e, \psi \rangle_{L^2(\pi)} &= \frac{1}{Z}\int_{\Omega}  e\cdot g^{-1}(h) \mathbb{D} \psi(h) \sqrt{\norm{g}(h)} \mathbb{P}(dh) \\
	&  \leq \left( \frac{1}{Z} \int_{\Omega} e \cdot g^{-1}(h) e \sqrt{\norm{g}(h)}\mathbb{P}(dh) \right)^{\frac{1}{2}} \Norm{\phi}_{\mathcal{H}_1} \\
	& \leq  \Norm{\phi}_{\mathcal{H}_1}.
	\end{split}
	\end{equation*}
It follows that $F^e \in \mathcal{H}_{-1}$ with $\Norm{F^e}_{\mathcal{H}^{-1}} \leq 1.$
\smartqed \qed
\end{proof}
$ $\linebreak
The $\lambda$-corrector $\chi_\lambda^e$ can be written as $$\chi^e_\lambda(h) = \int_0^\infty e^{-\lambda t}P(t) F^e(h),$$
and so by the contractivity of $P(t)$ we have that $\Norm{\chi_\lambda^e}_{L^2(\pi)} \leq \frac{1}{\lambda}\Norm{F^e}_{L^2(\pi)}$.
  Moreover,  taking the inner product of (\ref{eq:ergodic_cell_lambda}) with $\chi^e_\lambda$ we have that
\begin{equation*}
\begin{split}
  \lambda \Norm{\chi_\lambda^e}^2_{L^2(\pi)} + \Norm{\chi_\lambda^e}^2_{\mathcal{H}_1} 
 = \langle F, \chi^e_\lambda\rangle \leq \Norm{F}_{\mathcal{H}^{-1}}\Norm{\chi^e_\lambda}_{\mathcal{H}^1},	
\end{split}
\end{equation*}
so that that $\Norm{(\lambda I- \mathcal{L})^{-1}F^e}_{\mathcal{H}_1} \leq \Norm{\chi_\lambda^e}_{\mathcal{H}_{-1}}$.   Consequently,  we can extend the resolvent operator $(\lambda - \mathcal{L})^{-1}$ from $L^2(\pi)$ to a bounded operator from $\mathcal{H}_{-1}$ to $\mathcal{H}_{1}$.  To be able to obtain a central limit theorem one must show that the $\lambda$-correctors decay suitably fast in $L^2(\pi)$ as $\lambda \rightarrow 0$ and that $\chi^\lambda_e$ converges to an element in $\mathcal{H}_1$.  These two results are typically the core of any invariance principle for additive functionals of Markov processes.
\\
\begin{lemma}[\cite{kipnis1986central}, \cite{de1989invariance}]
	\label{lem:chi_lambda_limit}
	There exists $\chi^e \in \mathcal{H}_1$ such that 
	\begin{equation}
	\label{eq:ergodic_relation2}
		\lim_{\lambda \rightarrow 0}\Norm{\chi^e_\lambda - \chi^e}_{\mathcal{H}_1} =0 ,
	\end{equation}
and 
\begin{equation}
	\label{eq:ergodic_relation1}
	\lim_{\lambda\rightarrow 0 }\lambda \langle \chi^e_\lambda, \chi_\lambda^e \rangle_{L^2(\pi)} = 0
\end{equation}
\smartqed \qed
\end{lemma}

We can now state the homogenization theorem for $X^\epsilon_h(t)$.  The proof is a straightforward extension of the arguments given in \cite{kipnis1986central} or \cite{de1989invariance}.  An equivalent, but far more general, approach can be found in \cite{komorowski2012fluctuations}.  As in \cite{gonzalez2008first} we use the convention that $\left(\mathbb{D}\chi\right)_{ij} = D_j \chi_i$.
\\
\fxnote{This is the main theorem of the paper!}
\begin{theorem}
\label{thm:ergodic_homog_thm}
	Suppose that  conditions \textbf{A-D} hold. Then, the process $X_h^\epsilon(t)$ converges weakly in $C([0,T]; \R^d)$ to a Brownian motion with constant diffusion tensor $D$ given by:	
	\begin{equation}
	\label{eq:ergodic_eff_diff}
		D =  \frac{1}{Z}\int_{\Omega}\left(I + \mathbb{D} \chi(h)\right) g^{-1}(h)\left(I + \mathbb{D} \chi(h)\right)^\top \sqrt{\norm{g}(h)}\, \mathbb{P}(dh),
	\end{equation}
where $\chi = (\chi^{e_i})_{i=1,\ldots,d}$ is the $\mathcal{H}_1$ limit of $(\chi_{\lambda}^{e_i})_{i=1,\ldots, d}$ which exists by Lemma \ref{lem:chi_lambda_limit}.
\\
\smartqed \qed
\end{theorem}

\begin{corollary}
	Let $u^\epsilon(t,x ,h)$ be the solution to the Kolmogorov backward equation (\ref{eq:ergodic_pde}), with initial condition $v \in C_b(\R^d)$, independent of $\epsilon$.  Then 
	\begin{equation}
	\label{eq:ergodic_kbe_limit}
		\lim_{\epsilon\rightarrow 0}\mathbb{E}_{\mathbb{P}}\norm{u_h^\epsilon(t,x) - u^0(t,x)} = 0, \qquad	\mbox{ for all } (t,x)\in [0,T] \times \R^d,
	\end{equation}
where $u^0:[0,T]\times \R^d \rightarrow \R^d$ is the solution of 
\begin{equation}
	\label{eq:ergodic_limiting_pde}
	\frac{\partial u^0(t,x)}{\partial t} = D:\nabla \nabla u^0(t,x), 	\qquad (t,x ) \in [0,T]\times \R^d,
\end{equation}
where $D$ is given by (\ref{eq:ergodic_eff_diff}).
\end{corollary}
\smartqed \qed

\section{Properties of the Effective Diffusion Tensor}
\label{sec:properties}
In this section we study the properties of the effective diffusion tensor $D$ given by (\ref{eq:ergodic_eff_diff}).  For one dimensional surfaces, one can show that:
$$\mathbb{D}\chi(h) = \frac{\sqrt{\norm{g}(h)}}{Z} - 1,$$ so that $D = \frac{1}{Z^2}$,  where $Z$ is the average excess surface area given by (\ref{eq:Z}).  This generalises the corresponding result for lateral diffusion on a raplidy fluctuating periodic surfaces considered in \cite{duncan2013multiscale}.  As in the periodic case,  in two-dimensions or more it is not generally possible to obtain a closed form expression for $\mathcal{D}\chi$ and hence $D$.  This is compounded by the fact that $\chi$ is obtained as the limit of the $\lambda$-correctors $\chi_\lambda$ in the abstract space $\mathcal{H}_1$.
\\\\
In this section we show that $\mathcal{D}\chi$ can be expressed as the unique weak solution of a variational problem in $\left(L^2(\Omega)\right)^d$, and that the effective diffusion tensor $D$ can be identified as the minimum value of quadratic functional over the space of mean-zero, curl-free vector functions of $\Omega$.  Using this variational formulation one can easily obtain bounds on the effective diffusion tensor.  By considering the dual minimisation problem one can also obtain lower bounds for $D$. The approach taken here follows the exposition given in  \cite[Chapter 10]{komorowski2012fluctuations}.
\\\\
Denote by $(L^2(\mathbb{P}))^d$ the space of $\R^d$-valued functions of $\Omega$ with components in $L^2(\mathbb{P})$, equipped with the inner product
$$
	\langle U, V \rangle = \sum_{i=1}^d \langle U_i, V_i \rangle_{\mathbb{P}}.
$$
The gradient operator $\mathbb{D}$ defined in (\ref{eq:ergodic_gradient}) maps $H_1$ into $(L^2(\mathbb{P}))^d$.   Define $L^2_{pot}(\mathbb{P})$ to be the range of $\mathbb{D}$ in $(L^2(\mathbb{P}))^d$.   Let $L^2_c(\mathbb{P})$ be the space of constant vector fields in $(L^2(\mathbb{P}))^d$, that is
$$
	L_c^2(\mathbb{P}) = \mbox{span}\lbrace e_i  \, | \,i=1,\ldots,d \rbrace,
$$
where $e_i$ is the $i^{th}$ coordinate basis element of $\R^d$.  Finally, define $L^2_{div}(\mathbb{P})$ to be the orthogonal complement of $L_c^2(\P) \oplus L_{pot}^2(\P)$ in $(L^2(\mathbb{P}))^d$, so that we obtain the following Helmholtz decomposition
	$$ (L^2(\mathbb{P}))^d = L^2_{pot}(\P)\oplus L^2_{div}(\P)\oplus L^2_c(\P).$$
The space $L_{div}^2(\P) \oplus L^2_c{(\mathbb{P})}$ can be interpreted as the space of divergence-free vector fields with square integrable components.  The following result shows that $\mathbb{D}\chi$ can be expressed as the unique weak solution of a cell equation posed in $\left(L^2(\P) \right)^d$.  Note that in the case where the fluctuations are periodic this reduces to the ``periodic" cell problem.
\\
\begin{proposition}
	For any $e \in \R^d$ such that $\norm{e} = 1$,  $V = \mathbb{D} \chi^{e}$ is the unique solution of the problem
	\begin{equation}
	\label{eq:ergodic_cell_eqn}
		\begin{split}
		 V &\in L^2_{pot}(\P), \\
		\sqrt{\norm{g}(h)} g^{-1}(h)\left(e + V(h)\right) &\in L^2_{div}(\P).\\
		\end{split}
	\end{equation}
\end{proposition}
$ $\\
Analogously to the corresponding result for periodic surface fluctuations, given in \cite{duncan2013multiscale}, $D$ can be expressed as the minimum of a particular quadratic functional.   Indeed, if $e \in \R^d$ is a unit vector,  then the macroscopic rate of diffusion in the direction $e$ can be written as 
	\begin{equation}
		\label{eq:ergodic_minimisation}
		e \cdot D e = \frac{1}{Z}\inf_{V \in L^2_{pot}(\P)} \int_{\Omega} \left(e  + V(h)\right)\cdot g^{-1}(h)\left(e  + V(h)\right)\sqrt{\norm{g}(h)}\P(dh).
	\end{equation}
This can be seen by noting that that the weak cell equation (\ref{eq:ergodic_cell_eqn}) is the Euler-Lagrange equation (\ref{eq:ergodic_minimisation}), and that $\mathbb{D} \chi^e$ is the unique minimiser of this variational problem.  In particular,  by substituting $V = 0$ we obtain a (rough) upper bound for the effective diffusion tensor.
\\\\
One can also obtain a lower bound for $D$ simply by extending the domain over which (\ref{eq:ergodic_minimisation}) to $L^2_{pot}(\P) \oplus L^2_{div}(\P)$,  in particular:
$$e \cdot D e \geq \frac{1}{Z}\inf_{\substack{ V \in (L^2(h))^d, \\ \int V \P(dh) = 0}}\int_{\Omega} \left(e  + V(h)\right)\cdot g^{-1}(h)\left(e  + V(h)\right)\sqrt{\norm{g}(h)}\P(dh).$$
This minimisation problem can be solved directly to obtain a closed-form expression for the minimum value, giving the following lower bound.
$$
 e\cdot D e \geq e \cdot \frac{1}{Z}\left(\int_{\Omega} \frac{g(h)}{\sqrt{\norm{g}(h)}}\P(dh)\right )^{-1}\, e.
$$
We summarize the above properties of $D$ in the following theorem:
\fxnote{Is the fact that the effective diffusion coefficient depleted wrt microscopic diff coefficient given sufficient importance?}
\\
\begin{theorem}\label{thm:properties}  The effective diffusion tensor $D$ satisfies the following properties:	
	\begin{enumerate}
		\item $D$ is  strictly positive definite.
		\item For all $e \in \R^d$, $e \cdot D$ is given by:
		\begin{equation}
			e\cdot D e = \frac{1}{Z}\inf_{V \in L^2_{pot}(\P)} \int_{\Omega} \left(e + V(h)\right)\cdot g^{-1}(h)\left(e + V(h)\right)\sqrt{\norm{g}(h)}\,\mathbb{P}(dh),
		\end{equation}	
		 Moreover, $\chi$ is the unique minimiser of this functional.
		\item For all $e \in \R^d$, the effective diffusion tensor $D$ satisfies the following inequality:
			$$e\cdot D_* e \leq e \cdot D e \leq e \cdot D^* e ,$$
		where 
		\begin{equation}
			D^* = \frac{1}{Z}\int_{\Omega} g^{-1}(h)\sqrt{\norm{g}(h)}\P(dh),
		\end{equation}
		and
		\begin{equation}
			D_* = \frac{1}{Z}\left(\int_{\Omega} \frac{g(h)}{\sqrt{\norm{g}(h)}}\P(dh)\right)^{-1}.
		\end{equation}
	\item In particular $D$ satisfies:
	\begin{equation}
		\label{eq:ergodic_depletion}
		\frac{1}{Z^2} \leq e \cdot D e \leq 1.
	\end{equation}
	\end{enumerate}
\end{theorem}
\smartqed \qed
\begin{remark}
	In particular, Theorem \ref{thm:properties} implies that the macroscopic diffusion tensor $D$ is always depleted with respect to the microscopic diffusion tensor (which is rescaled to be $I$.  This is intuitively clear,  as we expect a particle undergoing Brownian motion along a surface to require extra effort to surpass surface undulation compared to a free Brownian on the underlying plane.   This is analogous to the case of diffusive transport of passive particles in a potential flow, where the macroscopic diffusion tensor is always depleted,\cite{pavliotis2008multiscale}.
\end{remark}

\section{The Area Scaling Approximation}
\label{sec:area_scaling}
In this section we derive a  closed-form expression for the effective diffusion tensor $D$ which holds for a large class of two-dimensional random surfaces.  More specifically, we show that if $D$ is isotropic,  then the \emph{area-scaling approximation} holds, namely that $D = \frac{1}{Z}$,  where $Z$ is the average surface area given by (\ref{eq:Z}).  This result generalises the corresponding result for the periodic surface case described in \cite{gustafsson1997diffusion, naji2007diffusion,granek1997semi} and proved rigorously in \cite{duncan2013multiscale}.  The result is based on a duality transformation argument similar to that described in \cite{kohler1982bounds} and  \cite[Section 1.5]{zhikov1994homogenization}, which relates the effective conductivity coefficient corresponding to the two-dimensional multiscale problem:
$$-\nabla\cdot\left(A^{\epsilon}(x)\nabla u^{\epsilon}(x)\right) = 0, \qquad x \in \Omega \subset \R^2$$
to the effective conductivity coefficient $A_Q$ arising from the ``rotated problem'':
$$-\nabla\cdot\left(Q^\top A^{\epsilon}(x)Q\nabla u^{\epsilon}(x)\right) = 0, \qquad x \in \Omega \subset \R^2$$
where $Q$ is a rotation about the origin.  A particular corollary of this argument is that the determinant of the conductivity coefficient is preserved in the limit as $\epsilon \rightarrow 0$,  that is,  if $A^\epsilon$ has determinant $k$ for all $\epsilon > 0$ then $A$ has determinant $k$ also.
\\\\
In two dimensions, the matrix $g^{-1}(x/\epsilon,h)\sqrt{\norm{g}(x/\epsilon,h)}$ has determinant $1$ for all $\epsilon > 0$.  By a straightforward modification of the arguments of  \cite[Theorem 1]{kohler1982bounds} we are able to provide a closed form expression for the determinant of the effective diffusion tensor, relating $\norm{D}$ to the average excess surface area $Z$.  In the particular case when $D$ is isotropic, we thus obtain an explicit formula for $D$.  This generalises the area scaling estimate described in \cite{gustafsson1997diffusion, naji2007diffusion,granek1997semi,duncan2013multiscale} for periodic surfaces to surfaces defined by stationary, ergodic random fields.
\\
\begin{theorem}
	\label{thm:as}
	In two dimensions, $D$ satisfies the following relationship
	\begin{equation}
			\label{eq:det}
			\det\left(D\right) = \frac{1}{Z^2}.
	\end{equation}
	Consequently, if $\lambda_1$ and $\lambda_2$ are the eigenvalues of $D$ with $\lambda_1 \leq \lambda_2$, then
	\begin{equation}
		\frac{1}{Z^2} \leq \lambda_1 \leq \frac{1}{Z} \leq \lambda_2 \leq 1.
	\end{equation}
	In particular, if $D$ is isotropic, then it can be written explicitly as
	\begin{equation}
			\label{eq:area_scaling}
			D = \frac{1}{Z}\mathbf{I}.
	\end{equation}
\end{theorem}
\fxnote{Proof is a modification of a proof by Papanicolau and Kohler.  Not sure if I should skip this?  I don't have this proof in the 1st paper.}
\begin{proof}
	We follow an approach similar to \cite{kohler1982bounds}.   We first note that Thompson's duality principle \cite[Section 2.6.2]{mei2010homogenization} applies equivalently in the space $\left(L^2(\P)\right)^d = L^2_{\rm pot}(\P) \oplus L^2_{\rm div}(\P)\oplus L^2_{\rm c}(\P)$, so that
\begin{equation}
	\label{eq:case1_dual_minimisation}
	e \cdot \left(ZD\right)^{-1} e = \displaystyle\inf_{F \in L^2_{div}(\P)}\int_{\Omega} \left(F(h) + e\right)\cdot \frac{g(h)}{\sqrt{\norm{g}(h)}}\left(F(h) + e\right) \, \P(dh)
\end{equation}
Let $Q:\R^2 \rightarrow \R^2$ denote a $\frac{\pi}{2}$ rotation about the origin in $\R^2$.    Given $F \in L^2_{div}(\P)$, define $$\mathcal{Q}F(h) =  (QF)(h).$$
The map $\mathcal{Q}:\left(L^2(\P)\right)^d \rightarrow \left(L^2(\P)\right)^d$  defined by
\begin{equation*}
	\mathcal{Q}G(h) = \left(QG\right)(h),
\end{equation*}
is an isomorphism between the sets
$$\left\lbrace \mathbb{D} f \, | \, f \in C^1_b(\Omega) \right\rbrace \mbox{ and }  \lbrace F \in (C^1_b(\Omega))^2  \, | \, \int_{\Omega}F(h)\P(dh) = 0 \mbox{ and } \mathbb{D}\cdot F = 0 \rbrace,$$
which can be extended to an isomorphism between $L_{pot}^2(\P)$ and $  L^2_{div}(\P)$.   Thus (\ref{eq:case1_dual_minimisation}) can be rewritten as
\begin{equation*}
\begin{split}
	e \cdot \left(Z \, D\right)^{-1} e &= \displaystyle\inf_{G \in L^2_{pot}(\P)}\int_{\Omega} \left(QG + e\right)\cdot \frac{g(h)}{\sqrt{\norm{g}(h)}}\left(QG + e\right) \, \P(dh) \\
	&= \displaystyle\inf_{G \in L^2_{pot}(\P)}\int_{\Omega} \left(G + Q^\top e\right)\cdot Q^\top \frac{g(h)}{\sqrt{\norm{g}(h)}}Q\left(G + Q^\top e\right) \, \P(dh).
\end{split}
\end{equation*}
However, in two dimensions,  for any invertible matrix $A$ we have that
\begin{equation}
	Q^\top A^{-1} Q = {A^\top}/{\det({A})}, 
\end{equation}
so that, since $\det\left(g^{-1}\sqrt{\norm{g}(y)}\right) = 1$,
\begin{equation*}
\begin{split}
	e \cdot \left(Z \, D \right)^{-1} e &= \displaystyle\inf_{G \in L^2_{pot}(\P)}\int_{\Omega} \left(G + Q^\top e\right)\cdot g^{-1}(h)\left(G + Q^\top e\right) \, \sqrt{\norm{g}(h)}\, \P(dh) 	\\
			&= \left(Q^\top e\right) \cdot Z\, D \left(Q^\top e\right).
\end{split}
\end{equation*}
Thus $$\frac{1}{Z}\, e \cdot D^{-1} e = Z\, e \cdot Q D Q^\top e  = Z\det({D})\, e \cdot D^{-1} \, e,$$
so that $\det(D) =  \frac{1}{Z^2}$. 
\\
\smartqed \qed
\end{proof}
$ $\\\\
Following the results of Theorem \ref{thm:as} it is natural to ask for conditions which guarantee that $D$ is isotropic.  By applying Schur's lemma  \cite{schur1905neue,alexanderian2012homogenization} we are able to provide a natural sufficient condition for $D$ to be isotropic.  To this end, let $Q \in \R^{2\times 2}$ be a proper orthogonal matrix.  Define $\mathcal{Q}^\top:\Omega \rightarrow \Omega$ to be
\begin{equation*}
	\mathcal{Q}^{\top}h(x)  = h(Q^{\top}x)	\qquad	x \in \R^{2}.
\end{equation*}
Clearly $\mathcal{Q}^{\top}$ is an isometry on $h$ which induces the following transformations on the metric tensor.
\\
\begin{lemma}
	\label{lemm:metric_tensor_rotation_stoch}
	Let $Q \in \R^{2\times 2}$ be any rotation about the origin, then
	\begin{equation}
		\label{eq:metric_tensor_condition_stoch}
				{g}^{-1}(x, \mathcal{Q}^{\top}h) = Q g^{-1}(Q^{\top}x,h) Q^{\top}
		\end{equation}
	and 
	\begin{equation}
		\label{eq:area_element_condition_stoch}
				\norm{{g}}(x, \mathcal{Q}^{\top}h) = \norm{{g}}(Q^{\top}x, h),
	\end{equation}
	for all $x \in \mathcal{D}$.
\end{lemma}
\begin{proof}
	It follows from the chain rule that 
	\begin{equation} 
		\mathbb{D} \left(\mathcal{Q}^\top h\right)(x)  = \nabla h\circ Q^{\top}(x) = Q\left(\mathbb{D} h\right) (Q^{\top}x).
	\end{equation}
From this, it is clear that 
\begin{equation*}
\begin{split}
	{g}(x, \mathcal{Q}^\top h) &= I + \mathbb{D} \left( \mathcal{Q}^\top{h}\right)(x) \otimes  \mathbb{D} \left( \mathcal{Q}^\top h)\right)(x) \\
					&= I + Q  \left[\left(\mathbb{D} {h}\right)(Q^\top x) \otimes  \left(\mathbb{D} {h}\right)(Q^\top x) \right] Q^\top \\
					&= Q \, {g}(Q^\top x, h) Q^\top.
\end{split}
\end{equation*}
\end{proof}
\\
We can now state the sufficient condition for the effective diffusion tensor to be isotropic.
\\

\begin{theorem}
	\label{thm:ergodic_isotropic}
	Let $Q \in \R^{2\times 2}$ be a rotation about some point by an angle not equal to $0$ of $\pi$.   Suppose that the random field measure $\P$ is invariant with respect to the corresponding operator $\mathcal{Q}^\top$,  that is $$\P \circ \left(\mathcal{Q}^{\top}\right)^{-1} = \P.$$   Then  $D$ is isotropic.
\end{theorem}
\begin{proof}
	By stationarity, we may assume that $Q$ is a rotation about the origin.   The set $\lbrace \mathbb{D} f \, | \, f \in C^1_b(\Omega) \rbrace$ is dense in $L^2_{pot}(\P)$, thus we may minimise (\ref{eq:ergodic_minimisation}) over this set.  Moreover, since $\mathcal{Q}^{\top}$ is measure-preserving we can make the substitution  $h \rightarrow \tau_x\mathcal{Q}^\top h$ in (\ref{eq:ergodic_minimisation}) to get 
	\begin{align*}
		e\cdot D e = \frac{1}{Z}\inf_{f \in C^1_b(\Omega)}\int_{\Omega}& \Big[\left(\nabla{{f}}(x, \mathcal{Q}^{\top}h) + e\right)\\ &\cdot {{g}^{-1}}(x, \mathcal{Q}^{\top}h)\left(\nabla{{f}}(x, \mathcal{Q}^{\top}h) + e\right) \sqrt{\norm{{{g}}}(x, \mathcal{Q}^{\top}h)}\Big]\,\mathbb{P}(dh),
	\end{align*}
Substituting (\ref{eq:metric_tensor_condition_stoch}) in the above we obtain
\begin{align*}
		e\cdot D e = \frac{1}{Z}\inf_{f \in C^1_b(h)}\int_{\Omega}&\Big[ Q^\top\left(Q\nabla{{f}}(Q^\top x, h) + e\right)\\ &\cdot {{g}^{-1}}(Q^\top x, h)Q^\top\left(Q\nabla{{f}}(x, h) + e\right) \sqrt{\norm{{{g}}}(Q^\top x, h)}\Big]\mathbb{P}(dh).
	\end{align*}
Using the fact that $Q$ is orthogonal and $\mathbb{P}$ is invariant under translations $\tau_y$ for any $y \in \R^2$ we obtain
\begin{equation*}
	\begin{split}
	e\cdot D e & = \frac{1}{Z}\inf_{f \in C^1_b(\Omega)}\int_{\Omega} \left(\mathbb{D}{{f}}(h) + Q^\top e\right)\cdot {{g}^{-1}}(h)\left(\mathbb{D}{{f}}( h) + Q^\top e\right) \sqrt{\norm{{{g}}}(h)}\mathbb{P}(dh)\\
		& = \left(Q^\top e\right)\cdot D \, \left(Q^\top e\right) \\
		&=  e \cdot \left(Q D Q^\top\right)e.
	\end{split}
\end{equation*}
Since $e$ is arbitrary, it follows that $D = QDQ^{\top}$ and so, by applying Schur's lemma it follows that the  effective diffusion $D$ is isotropic.
\smartqed \qed
\end{proof}
$ $\\\\
As an immediate corollary of Theorem \ref{lemm:metric_tensor_rotation_stoch} we note that it is sufficient that the random field is isotropic,  i.e. the two point covariance is of the form $C(x, y) = C(\norm{x-y})$ for $D$ to be isotropic.

\section{Numerical Scheme}
\label{sec:numerical_methods}
In general, when the effective diffusion tensor cannot be expressed in terms of a closed-form expression, one must resort to numerical methods to approximating $D$.   Unlike in the periodic case, the expression (\ref{eq:ergodic_eff_diff}) for $D$ does not lend itself to numerical approximation, due to the fact that the corrector $\chi$ exists only in the abstract space $\mathcal{H}_1$.  In this section we describe a widely applied scheme to numerically approximate $D$ making use of a periodic approximation \cite{owhadi2003approximation,bourgeat2004approximations}.
\\\\
For a fixed realisation $h$ of the random field and $R > 0$, the scheme is as follows:
\begin{enumerate}
\item Define ${F}_R(x, h)$ and ${\Sigma}_R(x, h)$ to be the ``periodized''coefficients given by
	\begin{equation}
		{F}_{R}(x, h) = F({(x \mbox{ \textbf{mod} } B_R)}, h), \mbox{ and } {\Sigma}_{R}(x, h) = \Sigma({(x \mbox{ \textbf{mod} } B_R)}, h),
	\end{equation}
where $F(x,h)$ and $\Sigma(x,h)$ are the drift and diffusion coefficients given by (\ref{eq:drift}) and (\ref{eq:diffusion}) respectively and where $B_R = [0,R]^d$.

\item Let $X_R(t)$ be the solution of the It\^{o} SDE $$X_R(t) = {F}_{R}(X_R(t), h) \,dt + \sqrt{2{\Sigma}_{R}(X_R(t), h)}\,dB(t),$$ and consider the corresponding periodic homogenization problem which gives rise to an effective diffusion tensor $D_R(h)$.   There are numerical approaches to computing the periodized effective diffusion tensor $D_R(h)$.   We adopt a PDE approach,   solving the corresponding periodic cell equation using a piecewise linear finite element scheme,  and using this solution to compute $D_R(h)$ via quadrature.  This approach is described in detail in \cite[Section 5.6]{duncan2013multiscale}.
\\\\
\end{enumerate}
By a simple modification the arguments given in \cite{owhadi2003approximation} and \cite{bourgeat2004approximations}, as $R \rightarrow \infty$, one can show that the periodic approximation $D_R(h)$ will converge to $D$ for $\mathbb{P}$ almost every $h \in h$.  As an illustration of the above numerical scheme we present two random surface models and explore the properties of $D$ using numerical simulations.

\subsection{The Random Protrusion Surface}

In the first example we consider the problem of lateral diffusion on a ``random protrusion surface" a two-dimensional random surface comprised of randomly distributed protrusions, represented as "bump" functions where the centers of the bumps are determined by a Poisson point process with constant intensity $\lambda$.  More specifically, we consider a surface which can be formally written as the the graph of 

\begin{equation}
	\label{eq:ergodic_h_poisson}
	{h}(x) = \sum_{i} f(x - x_i),
\end{equation}
where $\lbrace x_i \rbrace_{i \in \mathbb{N}}$ is a realisation of a Poisson point process and 
\begin{equation}
f(x) = 
\begin{cases} \alpha\exp\left(-\frac{1}{1 - x^2}\right) & \norm{x} < 1 \\
0 & \norm{x} \geq 1, \end{cases} 
\end{equation}
where $\alpha > 0$ is a constant amplitude. A realisation of this random field over the region $[0,20]^2$ is plotted in Figure \ref{fig:ergodic_poisson}.  We note that the inclusions are allowed to overlap. 
\\
\begin{figure}[htb]
\includegraphics[width=\textwidth]{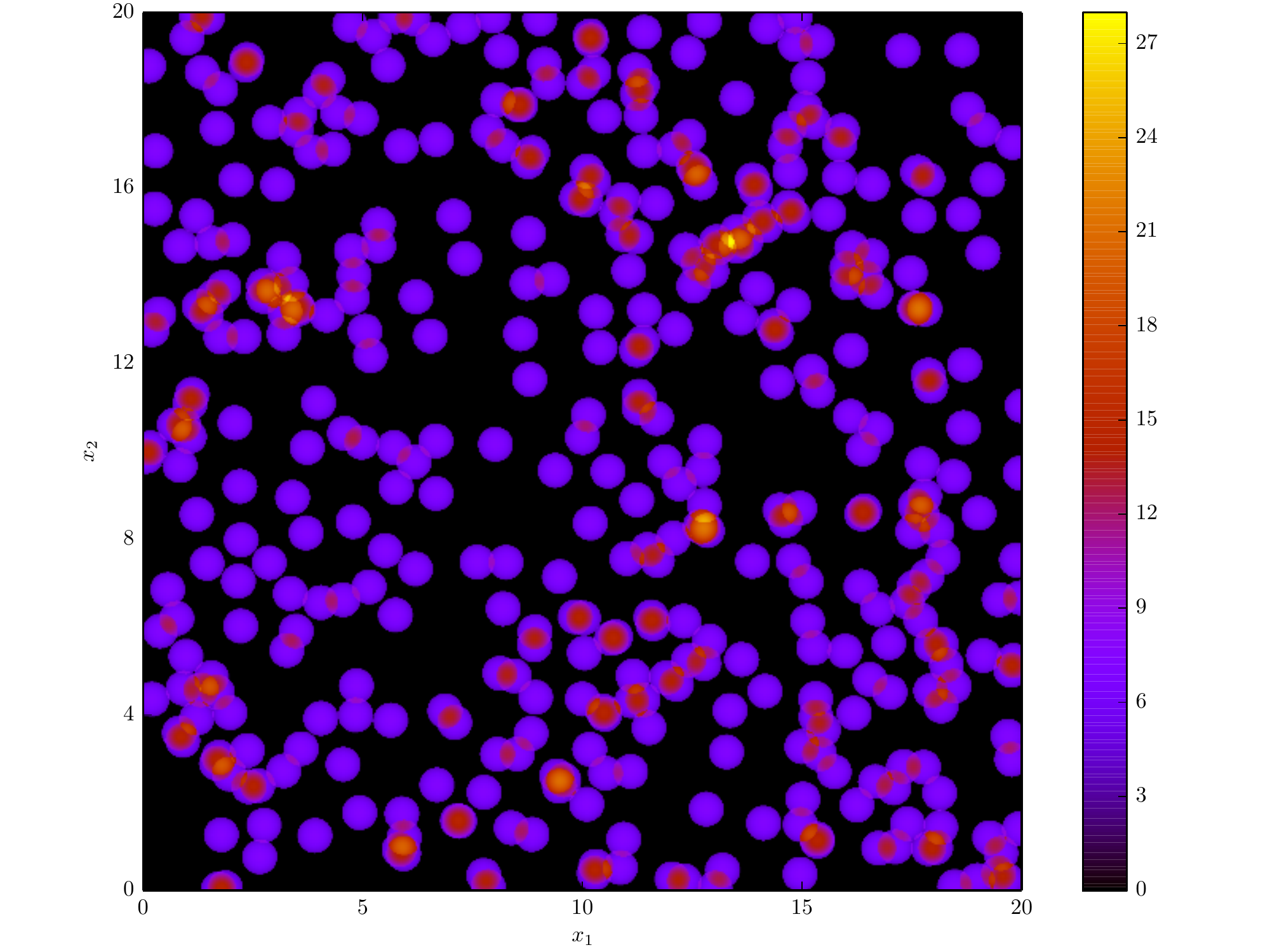}
\caption[Realisation of the ``random protrusion surface"]{Plot of a realisation of the random protrusion surface ${h}(x)$ with homogeneous intensity $\lambda = 1$, over the interval $[0,20]^2$.  Note the overlapping protrusions.}
\label{fig:ergodic_poisson}
\end{figure}

Similar models for random media are widely studied,  in particular in the study of random Schr\"{o}dinger operators \cite{pastur1971schrödinger,leschke2005survey}.    A Poisson point process with intensity $\lambda$ satisfies the following two fundamental properties \cite{daley2007introduction}:
\begin{enumerate}
	\item For every bounded, closed set $B$, the counting measure
	      $$N(B) := \norm{ \lbrace i \, : \, x_i(h) \in B \rbrace },$$  is a Poisson process distributed with mean $\lambda \mu(B)$, where $\mu(B)$ is the Lebesgue measure of $B$. 
	\label{cnd:ppp2}
	\item If $B_1, \ldots B_m$ are disjoint regions then $N(B_1), N(B_2), \ldots N(B_M)$ are independent.
	\\\\
\end{enumerate}
The Poisson point process is completely characterised by its Laplace functional, indeed if $\phi$ is a positive smooth function with compact support on $\R^2$ and we define
$$
	\nu(\phi) = \sum_{i}\phi(x_i(h)),
$$
then
\begin{equation}
	\label{eq:ergodic_laplace_functional_poisson}
	\mathbb{E}\left[ e^{-\nu(\phi)}\right] = \exp\left(\lambda \int e^{-\phi(y)} - 1 \, dy\right).
\end{equation}
From (\ref{eq:ergodic_laplace_functional_poisson}) we see that the Poisson point process is stationary with respect to spatial translations,  and thus so is ${h}(x)$.  Furthermore,  it is well known that ${h}(x)$ is ergodic with respect to spatial translations \cite[Proposition 2.6]{meester1996continuum}.  Realisations of the field ${h}(x)$ are clearly smooth and bounded with all derivatives bounded,  so that this random field satisfies the conditions of Theorem \ref{thm:ergodic_homog_thm}, which guarantees the existence of a homogenization limit.   Moreover,  it is straightforward to see that since the intensity $\lambda$ is constant, the conditions of Theorem \ref{thm:ergodic_isotropic} holds, and so $D$ is isotropic and thus equal to $\frac{1}{Z}$.   
\\\\
Properties 1 and 2 of the Poisson point process can be used to generate realisations of ${h}(x)$ over the domain $B_R = [0,R]^2$.  To sample the centers of the inclusions in this region, we first sample the number of points $N$ from the Poisson distribution with mean value $\lambda R^2$.   The centers of the inclusions $x_1(h), \ldots, x_N(h)$ are sampled uniformly in $[0,R]^2$.   
\\\\
To demonstrate the periodic approximation scheme,  in Figure \ref{fig:ergodic_poisson_eff_diff} we plot values of $D_R$ of the effective diffusion tensor for the random protrusion model,  for varying $R$ and for two sets of parameters, namely $\lambda = 0.5$, $\alpha = 1$ and  $\lambda = 1.5$, $\alpha = 1$.  Since $D_R$ quickly becomes isotropic as $R$ increases, we only show the first component.  For each value of $R$, $10^3$ independent surface realisations are generated,  and for each realisation,  $D_R$ is computed using a piecewise linear finite element scheme, refining the mesh-size until the relative error between successive refinements is $10^{-2}$.   The dashed lines denote the area scaling approximation of $D$, given by $D_{as} = \frac{1}{Z}$, and we see that there is good agreement between the mean value of $D_R(h)$ and $D$ for large values of $R$.
\\\\
To further confirm the results of Theorems \ref{thm:ergodic_homog_thm} and \ref{thm:as}, we compare the area-scaling estimate for this surface to the macroscopic diffusion tensor estimated from a long time MCMC simulation of a particle undergoing Brownian motion on a single realisation of the surface, using an Euler-Maruyama discretisation of (\ref{eq:ergodic_sde}).  In Figure \ref{fig:ergodic_mcmc_poisson}  we plot the macroscopic diffusion tensor computed for surfaces with parameters $\lambda = 0.5, \alpha=1$ and $\lambda = 1.5$ and $\alpha = 1$, respectively.  The particle trajectory is simulated  with timestep length $10^{-6}$ for $t \leq 4000$.  The macroscopic diffusion tensor $D_{mcmc}$ is computed ergoically from a single run of the Markov process using a sampling time-step of size $1$.  The dashed lines denote the area scaling approximation,  and we see that, as time increases,  the long term diffusion coefficient converges to the area-scaling estimate.   We note that while approximating the effective diffusion tensor directly from a Monte-Carlo simulation is far more straightforward than using the finite-element approach adopted here,  the latter method is more robust  and allows one to explore parameter regimes where surface realisations are possess rapid variations.  For such surfaces,  the resulting SDE becomes increasingly stiff and one must take increasingly smaller time-steps to correctly capture the long-term diffusion tensor,  which quickly becomes prohibitively expensive in terms of computation time.

\begin{figure}[htb]
\centering
\includegraphics[scale=0.5]{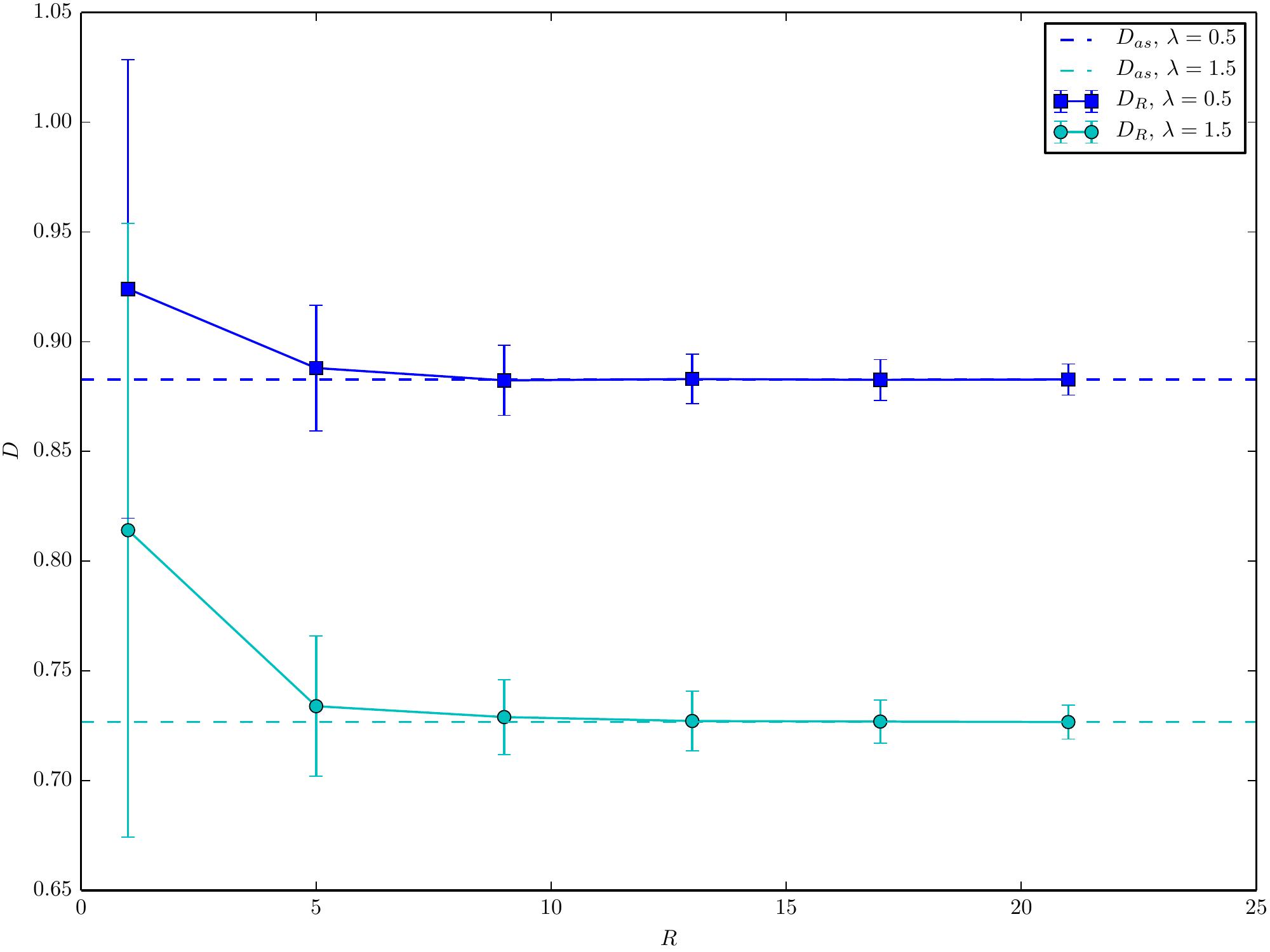}
\caption[]{Plots of the distribution of the first component $D_R(h)$ for varying $R$, for random protrusion surfaces with parameters $\lambda = 0.5, \alpha = 1.0$ (square markers) and $\lambda = 1.5, \alpha = 1.0$ (circle markers).  Error bars denote one standard-deviation of the distribution of $D_R(h)$, generated from $10^3$ surface realisations.  The dashed line indicates the value of the area-scaling estimate $\frac{1}{Z}$.}
\label{fig:ergodic_poisson_eff_diff}
\end{figure}

 \begin{figure}[htb]
\centering
	\includegraphics[scale=0.5]{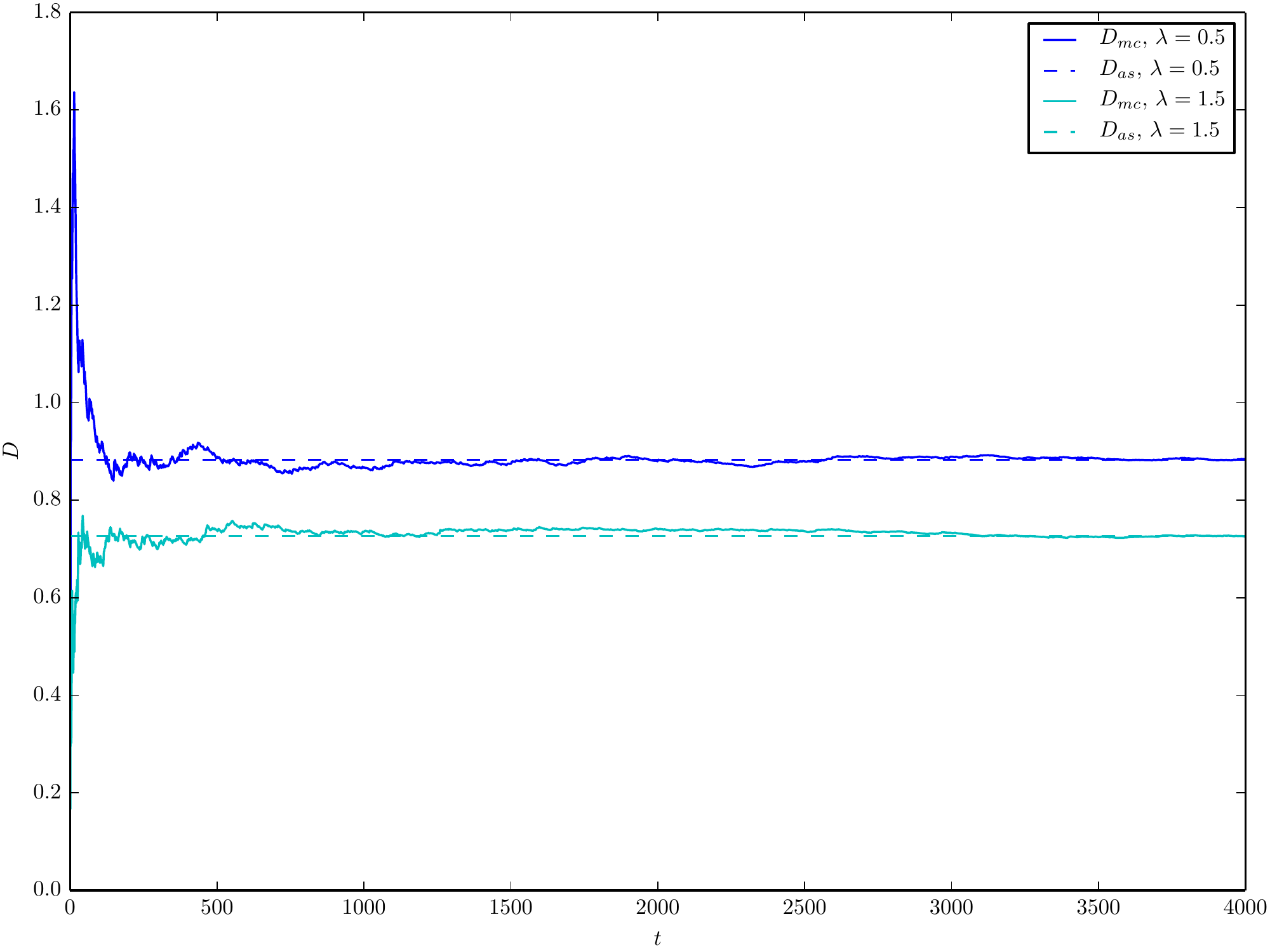} 
\caption[]{Plot of the first component of the macroscopic diffusion tensor computed from a long-time MCMC simulation of a Brownian motion on a single realisation of random protrusion surfaces with parameters $\lambda = 0.5$ and $\lambda = 1.5$ respectively.}
\label{fig:ergodic_mcmc_poisson}
\end{figure}

\subsection{Gaussian Random Field Surface}
\label{sec:ergodic_example2}
The second example we consider is a surface generated by a two-dimensional stationary Gaussian random field.  Due to the unbounded support of the random field fluctuations,  this case does not fall into the framework of this paper,  however, numerical experiments suggest that a homogenization limit does exist for lateral diffusion on such a surface and that the conclusions of Theorems \ref{thm:ergodic_homog_thm} and  \ref{thm:as} appear to still hold in this case.
\\\\
We consider an isotropic  Gaussian random field ${h}:\R^2 \rightarrow \R$ with mean zero and exponentially decaying autocorrelation given by $c_\alpha(r) = e^{-\pi\alpha\norm{r}^2}$, where $\alpha$ is a positive constant.  By  Bochner's theorem \cite[Theorem IX.9]{reed1975methods}, the function $c_\alpha(x-y)$ defines a covariance operator $C_\alpha$, and a Gaussian measure on $L^2(\R^d)$ with mean $0$ and covariance $C_\alpha$.  Moreover, by application of the Sobolev embedding theorem one can see that realisations of $h(x)$ have an almost-surely smooth modification.
\\
\begin{figure}[ht]
\includegraphics[width=\textwidth]{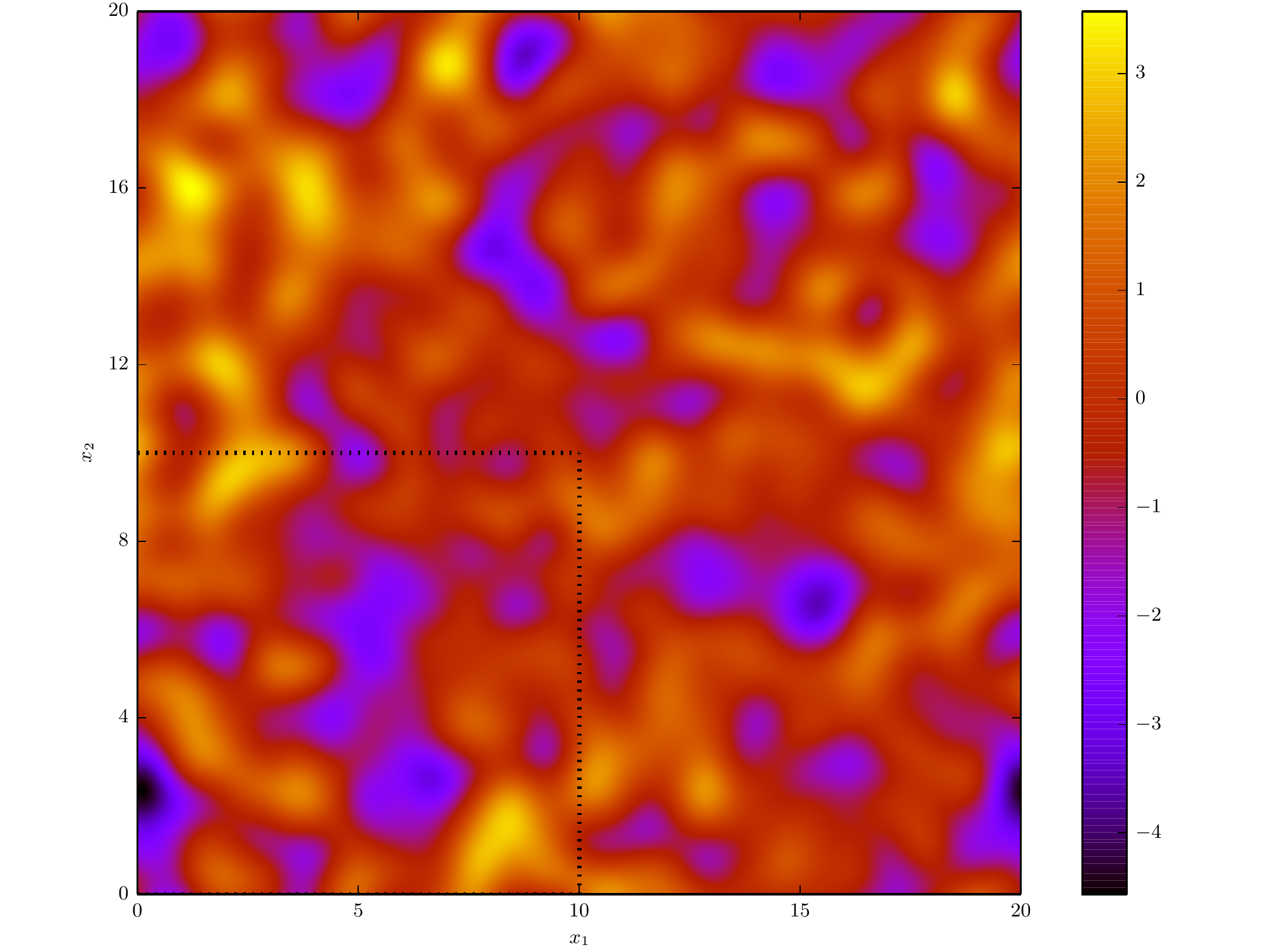}
\caption[Realisation of the surface generated by stationary Gaussian random field.]{A realisation of the Gaussian random field $h(x)$ using a truncated Karhunen-Loeve expansion,  with $\alpha = 1$, $M=1024$ and $R = 5$ (Note that the field has been translated periodically from $[-2R,2R]^2$ to $[0,4R]^2$).   The region enclosed by the dotted line is what is retained as a sample of ${h}(x)$.}
\label{fig:ergodic_grf}
\end{figure}

To simulate realisations of $h(x)$ over a domain $B_R = [-R,R]^2$, we make use of the Karhunen-Loeve expansion \cite[Chapter 3]{adler2007random} of the random field with respect to the standard Fourier basis in the space of periodic square-integrable functions on $[-2R,2R]^2$.  Given a realisation $h_{per}(x)$ in this space,  the random field $h(x)$ is then  approximated by $h(x) = \restr{h}{B_R}(x)$, provided $R$ is sufficiently large so that
\begin{equation}
	\label{eq:ergodic_decorrelation}
c_\alpha(r) \approx 0 \mbox{ for  } \norm{r} > R.
\end{equation}  
\\\\
For a given realisation of the surface $h(x)$ we use the periodic approximation scheme to compute $D_R(h)$ for varying $R$ and parameters $\alpha = 0.01$ and $\alpha = 0.1$, respectively.  As before,  for each value of $R$, $10^3$ realisations of the surface are generated and the periodic approximation $D_R$ computed for each realisation.   A starting mesh-size of $2^{-6}$ is used, refining globally until the relative error of $D_R(h)$ between successive refinements is $10^{-2}$.  As $R$ increases,  the variance of the samples of $D_R$ decreases, the ergodic average converges very quickly.  Indeed for $R \geq 10$ the ergodic average converges to the mean after only $50$ iterations.   As noted in the previous example however, this comes at the cost of requiring smaller mesh-sizes to maintain a constant  error for the finite element approximation as $R$ increases. In Figure \ref{fig:ergodic_grf_eff_diff}, for each $R$ we plot the average value of first component of $D_R(h)$.  We note that there is good agreement between the mean value of $D_R(h)$ and the effective diffusion tensor predicted by the area scaling approximation for large values of $R$. 
\\\\
In Figure \ref{fig:ergodic_mcmc_gaussian} we compare the area-scaling estimate with the first component of the macroscopic diffusion tensor $D_{mcmc}$ computed from a long-time simulation of Brownian motion on a single realisation of the Gaussian random surface, directly simulated from the SDE (\ref{eq:ergodic_sde}) using an Euler-Maruyama discretisation, with timestep $10^{-7}$.  Once again,  the macroscopic diffusion tensor is well approximated by the area-scaling approximation.  The results plotted in Figures \ref{fig:ergodic_mcmc_gaussian} and \ref{fig:ergodic_grf_eff_diff} suggest that the conclusions of Theorems \ref{thm:ergodic_homog_thm}, \ref{thm:ergodic_isotropic} and  \ref{thm:as} appear to hold true for the case of a Gaussian random field despite the fact that the homogenization theorem requires the assumption uniform bounds on the field and its derivatives.

\begin{figure}[htb]
\centering
\includegraphics[scale=0.5]{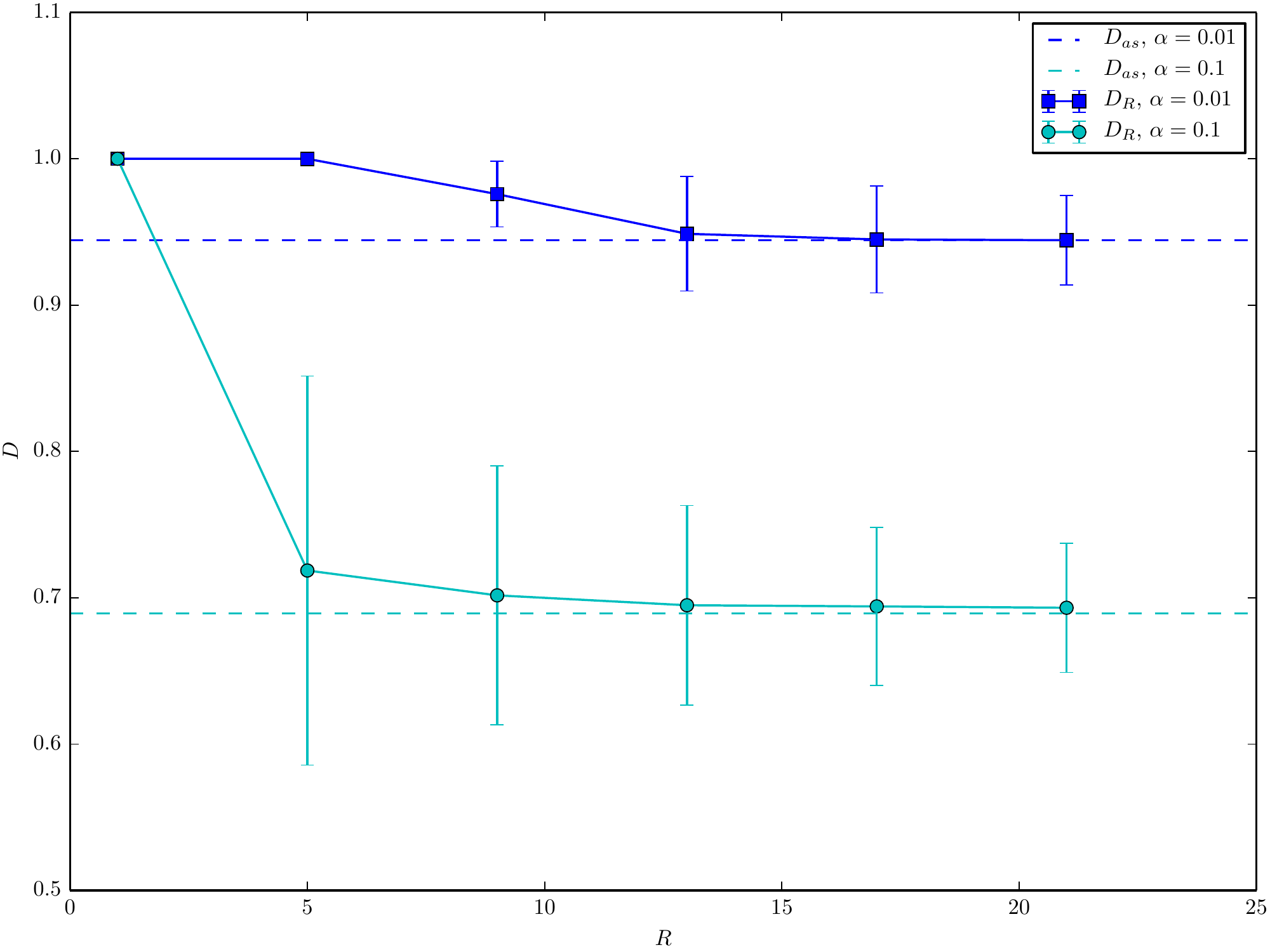}
\caption[]{A plot of the distribution of $D_R(h)$ for increasing values of $R$, for the Gaussian random field surface, with parameters $\alpha = 0.01$ (square markers) and $\alpha = 0.1$ (circle markers) respectively.   For each value of $R$, $10^3$ realisations were generated.  The error bars denote one standard deviation.  The dashed line indicates the  value of the area scaling estimate $D_{as}$ given by $\frac{1}{Z}$.}
\label{fig:ergodic_grf_eff_diff}
\end{figure}

 \begin{figure}[htb]
\centering
	\includegraphics[scale=0.5]{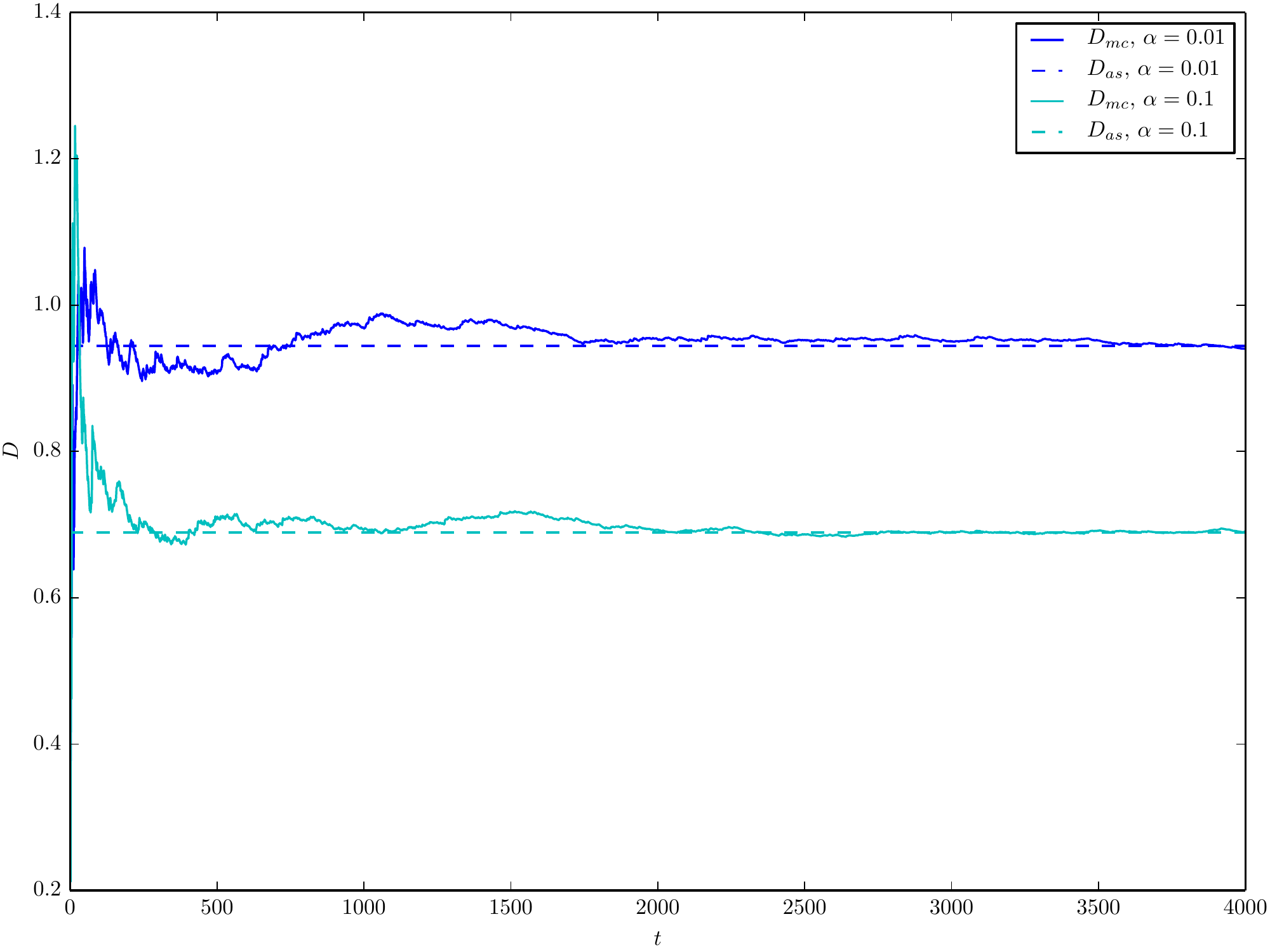} 
\caption[Macroscopic diffusion tensor computed from a long-term simulation of a Brownian motion on a realisation of the Gaussian random surface for $\lambda = 0.5$ and $\lambda = 1.5$.]{Plot of the macroscopic diffusion tensor computed from a long-time simulation of a Brownian motion on a realisation of the Gaussian random surface with parameters $\alpha = 0.01$ and $\alpha = 0.1$, respectively. }
\label{fig:ergodic_mcmc_gaussian}
\end{figure}

\section{Conclusion}
\label{sec:conclusion}
In this paper we have studied the problem of diffusion on a quasi-planar surface defined by a random field which is stationary and ergodic with respect to spatial translations.  We have shown that the problem of computing the effective dynamics can be expressed as a stochastic homogenization problem,  and subject to suitable conditions on the random field, we have applied standard results to show that the lateral diffusion process is well-approximated by a Brownian motion on the plane, with constant effective diffusion tensor $D$, independent of the particular surface realisation.  Although $D$ does not generally have a closed form, we have been able to identify a number of properties of the effective diffusion tensor.  In particular, we have obtained variational bounds on $D$,  showing that it is depleted with respect to the microscopic diffusion tensor.  Moreover,  we have been able to show that for two dimensional surfaces, the area scaling approximation $D = \frac{1}{Z}$ holds for isotropic $D$, and provided a natural sufficient condition on the random field for $D$ to be isotropic.  We have also described a practical numerical scheme to approximate the effective diffusion tensor using a periodic approximation,  and used this method to consider two very simple examples.
\\\\
The macroscopic behaviour of lateral diffusion on static surfaces with random fluctuations has been studied before in the context of modelling protein diffusion on Helfrich-elastic surfaces with quenched fluctuations, \cite{naji2007diffusion,gustafsson1997diffusion,duncan2013multiscale}.  However, these papers have all assumed that the random surface is periodic in each direction, with period length $L$ characterising the macroscopic length scale of the model. The long-time/macroscopic limit, computed via periodic homogenization, captures the dynamics of the particle which diffuses over a periodic repetition of a single realisation of the random surface.  The results in this paper characterize the macroscopic behaviour of diffusion on random surfaces without imposing any such periodization.  The existence of a homogenization limit is entirely due to the stationarity and ergodicity of the random surface and not any imposed periodicity.  Moreover, the resulting macroscopic limit depends only on the quantitative statistical properties of the random field, and independent of the particular surface realisation (unlike in the periodic case).   
\\\\
There are several extensions to the present work.  Clearly, as in the periodic case, it would be interesting to study the more general problem where the surface possesses a slowly varying component, and the rapid fluctuations occur normally to this slow surface.  The problem of finding the effective behaviour would result in a locally-stationary homogenisation problem as was considered in \cite{rhodes2009homogenization}.     More generally, it would be interesting to extend the approach to study more general surfaces, possibly even closed surfaces embedded in $R^3$ , which to our knowledge has not been previously considered.
\\\\
Another direction of interest would be to relax Assumption $\mathbf{D}$, namely the requirement that realisations of the field and its derivatives must be uniformly bounded.  Relaxing this assumption would permit one to obtain analytical results for Gaussian random fields.  Removing Assumption  $\mathbf{D}$ introduces several technical issues:  the crux of the problem lies in the fact that the drift of the SDE (\ref{eq:ergodic_sde_rescaled}) is no longer bounded, and the diffusion term no longer remains uniformly elliptic.  The issue of unbounded coefficients might be resolvable by adopting an approach similar to \cite{avellaneda1991integral, oelschlager1988homogenization,komorowski2003sector}, however it is still unclear how to handle the lack of ellipticity.   Nonetheless, numerical results suggest that a homogenization limit for Gaussian random fields exists,  and thus we believe that it is possible to obtain a homogenization result for such surfaces and leave the problem of proving this rigorously as scope for future work. 

\section*{Acknowledgements}
The author is grateful to Grigorios Pavliotis and Andrew Stuart for useful suggestions and comments.  Moreover,  the author wishes to acknowledge  EPSRC for financial support and thanks the Centre for Scientific Computing at Warwick for computational resources. 

\bibliographystyle{plain}
\bibliography{refs}
\end{document}